%% file: main_arXiv.tex
\documentclass[%
 aip,
 amsmath,amssymb,
preprint,
]{revtex4-1}

\usepackage{graphicx}
\usepackage{dcolumn}
\usepackage{bm}

\usepackage[utf8]{inputenc}
\usepackage[T1]{fontenc}
\usepackage{mathptmx}
\usepackage{etoolbox}
\usepackage{amsthm}
\usepackage{color}

\makeatletter
\def\@email#1#2{%
 \endgroup
 \patchcmd{\titleblock@produce}
  {\frontmatter@RRAPformat} {\frontmatter@RRAPformat{\produce@RRAP{*#1\href{mailto:#2}{#2}}}\frontmatter@RRAPformat}
  {}{}
}%
\makeatother

\newtheorem{theorem}{Theorem}[section]
\newtheorem{corollary}[theorem]{Corollary}
\newtheorem{lemma}[theorem]{Lemma}
\newtheorem{proposition}[theorem]{Proposition}
\newtheorem{definition}[theorem]{Definition}
\theoremstyle{definition}
\newtheorem{example}[theorem]{Example}
\newtheorem{remark}[theorem]{Remark}
\newtheorem{assumption}{Assumption}

\newcommand{\commentCS}[1]{#1}
\newcommand{\new}[1]{#1}

\newcommand{\R}{\mathbb{R}}
\newcommand{\C}{\mathcal{C}}
\renewcommand{\H}{\mathcal{H}}
\newcommand{\B}{\mathcal{B}}
\newcommand{\SpanX}{\mathrm{Span}\{k(x,\cdot) : x \in X\}}

\newcommand{\N}{\mathbb{N}}

\newcommand{\Kf}{\overline{K}_f}
\newcommand{\CN}{\mathbb{C}}

\renewcommand{\L}{\mathrm{L}}
\newcommand{\RKBS}{\left(\B,\B',\langle\cdot,\cdot\rangle,k\right)}

\begin{document}

\preprint{}

\title[Koopman and Perron-Frobenius Operators on reproducing kernel Banach spaces]{Koopman and Perron-Frobenius Operators on reproducing kernel Banach spaces}
\author{Corbinian Schlosser}
 \email{cschlosser@laas.fr}
\affiliation{LAAS-CNRS, 7 avenue du colonel Roche, F-31400 Toulouse; France}%

\author{Isao Ishikawa}%
 \email{ishikawa.isao.zx@ehime-u.ac.jp}
\affiliation{Center for Data Science, Ehime University, Matsuyama, 790-8577, Japan}%
\affiliation{Center for Advanced Intelligence Project, RIKEN, Tokyo, 103-0027, Japan}

\author{Masahiro Ikeda}
 \affiliation{Center for Advanced Intelligence Project, RIKEN, Tokyo, 103-0027, Japan}
 \affiliation{Department of Mathematics, Keio University, Yokohama, 223-8522, Japan}
 \email{masahiro.ikeda@riken.jp}
 
\date{\today}

\begin{abstract}
Koopman and Perron-Frobenius operators for dynamical systems have been getting popular in a number of fields in science these days. Properties of the Koopman operator essentially depend on the choice of function spaces where it acts. Particularly the case of reproducing kernel Hilbert spaces (RKHSs) draws more and more attention in data science. In this paper, we give a general framework for Koopman and Perron-Frobenius operators on reproducing kernel Banach spaces (RKBSs). More precisely, we extend basic known properties of these operators from RKHSs to RKBSs and state new results, including symmetry and sparsity concepts, on these operators on RKBS for discrete and continuous time systems.
\end{abstract}

\maketitle



\section{Introduction}
Koopman operators, as well as reproducing kernel Hilbert spaces, have been getting popular in various fields in science these days. Both contribute widely to applications\cite{AppliedKoopmanism, EDMD, KoopmanControlMilan, SS16, RKHSGlobalOptimization} but provide also theoretical insights to several fields\cite{EngelNagel, MezicMauroy, SS16, RKHSProbability} and many others.
Koopman operators and reproducing kernel Hilbert space techniques aim at translating the corresponding problems into a functional analytic setting and borrowing methods from there.

The Koopman (or composition) operator for a function $f:X \rightarrow X$ on a set $X$ is defined by $Tg := g \circ f$ for functions $g:X \rightarrow \mathbb{C}$ in a suitable function space. This lifting procedure results in a linear operator and the main idea is to transfer properties of the dynamical system to properties of the Koopman operator and vice versa.

The adjoint of the Koopman operator is called the Perron-Frobenius operator. The properties of the dynamical system that can be observed via the Koopman or Perron-Frobenius operator depend strongly on the choice of function space. In the study of topological properties of dynamical systems, the Koopman operator on the space of continuous functions was extensively studied.\cite{KariKuester,AppliedKoopmanism} For ergodic theory, the space $\mathrm{L}^2(\mu)$ of square integrable functions with respect to an invariant measure $\mu$ appears naturally and the Koopman operator has proven to be a powerful tool to enable the use operator theory.\cite{OperatorTheoreticAspectsOfErgodicTheory} For instance, spectral theory lies at the core of forecasting via dynamic mode decomposition. With regard to applications the space $\mathrm{L}^2(\mu)$ requires cautious treatment\cite{MilanDataDrivenKoopman} because it typically lacks the property of bounded (or even well defined) point evaluation, which for applications causes sensitivity with respect to the measurement. An RKHS on the other hand provides such a Hilbert space of functions with bounded point evaluations and additionally provides explicit access to the geometry via the kernel function. In machine learning these advantages are already widely used and have shown strong applications\cite{HIIK,SVM2008, KernelsInMachineLearning}. Recently a different application of Koopman theory to machine learning was found in speeding up the learning process\cite{dogra2020optimizing}.

To connect Koopman theory and RKBS methods, we follow the path in this direction explored for RKHS by\cite{Kawahara16, IFIHK, IIS20, AA17, CompOpAnalyticFunctions, Das, DasSpectrum, das2021reproducing,Rosenfeld,alexander2020operator, lian2021koopman,KernelMeanEmbedding,mezic2020spectrum}, treating several different problems. The advantage of working with an RKBS or RKHS is that it provides continuous point evaluation, i.e. robustness with respect to measurements. Secondly, it incorporates directly the underlying geometry of the RKBS leading to the fact kernel DMD\cite{Kawahara16} is the least-distance projection of the Perron-Frobenius operator on the space operators on the kernel functions\cite{Kawahara16}. \commentCS{Thirdly, in some cases the RKBS appears naturally via spectral decomposition of the Koopman operator\cite{mezic2020spectrum}.} This property flourishes when the RKBS can be chosen in such a way that its geometry has a desirable meaning for the dynamical system.

Our goal consists of providing a common framework for Koopman operators on reproducing kernel Banach spaces (RKBS), including RKHS, with a special focus on  boundedness (or at least closedness) of these operators. As in the RKHS case a dual perspective via the Perron-Frobenius operator brings advantages, in particular via its natural action on kernel functions in the RKBS. In Section \ref{sec:KoopmanOpAndRKBS} we define the central objects: The Koopman and Perron-Frobenius operator and reproducing kernel Banach spaces. In Section \ref{sec:Koopman_Op_RKBS} we introduce the Koopman and Perron-Frobenius operators on RKBS and, on one hand, we list and extend known general results for these operators resp. semigroups on RKHS to RKBS. On the other hand we give new results such as that under mild assumptions on the RKBS the Perron-Frobenius operator is not closed (Theorem \ref{ElementaryPropertiesKoopmanRKHS} 5.), but closeable whenever the Koopman operator is densely defined (under a reflexivity assumption on the RKBS, Theorem \ref{ElementaryPropertiesKoopmanRKHS} 6.). This is followed by some examples of Koopman and Perron-Frobenius operators on specific RKBS, including the fundamental case of linear systems, treatment of the space of continuous functions as RKBS and an application to the transport equation via Besov-spaces as underlying RKBS. In Section \ref{sec:ContinuousTime} we focus on continuous time systems. This leads to one-parameter semigroups of Koopman respectively Perron-Frobenius operators. We introduce their generators, present certain elements that belong to their domain and give a geometric condition for continuous time dynamical systems under which the Perron-Frobenius semigroup is uniformly bounded (hence so is the Koopman semigroup) and strongly continuous (Proposition \ref{prop:LumerPhillipsDirectApplication}). With a view to applications and computational aspects, we show in Section \ref{sec:SymmetrySparsity} that symmetry and sparsity of the dynamics can be preserved for our approach under corresponding properties of the RKBS.
    
Our investigations indicate three fundamental obstacles; first, the notion of RKBS is very general and, as a result, general results on Koopman and Perron-Frobenius operators are limited, second, not always can we let go of the regularity that comes with reproducing kernel Hilbert spaces such as reflexivity and explicit expression of the geometry in terms of the kernel and third, the kernel should be adapted to the dynamics in order to assure that the Koopman and Perron-Frobenius operators can flower out their potential.

\section{Notations}
The set of natural numbers is denoted by $\N$. For a complex number $a \in \CN$ we denote by $\overline{a}$ its complex conjugate and by $\mathcal{Re}(a)$ its real part. $\R_+$ denotes the non-negative real line $[0,\infty)$. By $\overline{B_1(0)}$ we denote the unit disc in $\CN$. The dual of a topological vector space $Y$ is denoted by $Y^*$. The domain of an operator $T$ is denoted by $D(T)$. By $T \subset S$ for operators $T,S$ we mean that $D(T) \subset D(S)$ and $Tx = Sx$ for all $x \in D(T)$. The adjoint of an operator $T$ is denoted by $T^*$ while the adjoint operator with respect to a given bilinear form is denoted by $T'$. For a compact topological space $X$, the space of continuous functions on a set $X$ is denoted by $\C(X)$ and equipped with the supremum norm. We identify the dual space of $\C(X)$ with the space of bounded Borel measures, denoted by $M(X)$. The dirac delta at a point $x \in X$ is denoted by $\delta_x \in M(X)$. For a measure $\mu$ on $X$ with corresponding sigma algebra $\Sigma$ and for $1\leq p \leq \infty$ we denote by $\mathrm{L}^p(X) = \mathrm{L}^p(X,\Sigma,\mu)$ the space of $p$ times Lebesgue integrable functions.

\section{Koopman operators and reproducing kernel Banach spaces}\label{sec:KoopmanOpAndRKBS}

In this section, we provide the definition of Koopman and Perron-Frobenius operators as well as the definition of reproducing kernel Banach spaces and some of their properties that we need in the following.

\subsection{Review of Koopman and Perron-Frobenius operators}
We start with a review of well known results on Koopman operators and point out some properties of the Koopman operator that motivate the use of RKBS for Koopman operators but also demonstrate that choosing a function space for the Koopman operator is a delicate task. This problem is to be expected since using the Koopman operator just means looking at the dynamical system from another perspective -- so the complexity does not change but the idea is to enable different tools from functional analysis.
\begin{definition}\label{Def:Koopman}
	Let $X$ be a set and $f:X \rightarrow X$. Let $Y$ be a (normed) function space on $X$. The Koopman operator $U_f:D(U_f)\rightarrow Y$ is given by
	\begin{equation}\label{equ:DefKoopman}
		U_fg := g \circ f
	\end{equation}
	where its domain $D(U_f) \subset Y$ is given by $D(U_f) = \{g \in Y: g\circ f \in Y\}$. If $U_f$ is densely defined its adjoint operator $K_f:D(K_f) \rightarrow Y^*$ with domain $D(K_f) \subset Y^*$ is called the Perron-Frobenius operator.
\end{definition}

\begin{remark}
    In case the underlying function space is $\mathrm{L}^p(X)$ for some $p\geq 1$ with respect to a given measure $\mu$ on $X$ then, for the Koopman operator to be well defined, $f$ has to satisfy $\mu(f^{-1}(N)) =0$ whenever $\mu(N) = 0$.
\end{remark}

\begin{remark}\label{rem:ContinuousTimeRemark}
    The definition of Koopman and Perron-Frobenius operators for continuous time systems is analogous, which we treat in Section \ref{sec:ContinuousTime}.
\end{remark}

\noindent
The Koopman operator is always linear and an interplay between the dynamics and the chosen function space can allow detailed descriptions of the Koopman operator as well as the dynamics.\cite{KariKuester, OperatorTheoreticAspectsOfErgodicTheory} 
For applications of the Koopman operator for problems from engineering the paper\cite{AppliedKoopmanism} was a seminal work.

\begin{example}\label{example:KoopmanCXL2}
We will mention two classical examples of function spaces the Koopman operator acts on $Y = \C(X)$ and $Y = \mathrm{L}^2(\mu)$.
\begin{enumerate}
	\item $Y = \C(X)$ for $X$ compact: If $f$ is continuous then $D(U_f) = \C(X)$ and hence $D(K_f) = \C(X)^* = M(X)$ where $M(X)$ denotes the space of signed Borel measures on $X$. It's adjoint operator $P_f: M(X) \rightarrow M(X)$, the Perron-Frobenius operator, satisfies
	\begin{equation}\label{eq:KoopmanPerronFrobAdjoint}
	    \int\limits_X g \circ f \; d\mu = \int\limits_X g \; dP_f\mu
	\end{equation}
	for any $g \in \C(X)$ and $\mu \in M(X)$. The Perron-Frobenius operator can be given more explicitly by the pushforward $P_f = f_\#$, i.e.
	\begin{equation}\label{equ:DefPerronFrobenius}
		P_f\mu(A) = \mu(f^{-1}(A))	
	\end{equation}
	for any $\mu \in M(X)$ and Borel set $A \subset X$.
The Koopman operator $U_f$ is a contractive linear algebra homomorphisms, i.e. $U_f$ is a linear operator with $\max\limits_{x \in X} |U_fg(x)| \leq \max\limits_{x \in X} |g(x)|$ and $U_f(g_1 \cdot g_2) = U_fg_1 \cdot U_f g_2$ for all $g_1,g_2 \in \C(X)$. Apart from these intriguing properties, the Koopman operator $U_f$ on $\C(X)$ has some disadvantages. Among these are that as long as the image of $f$ contains infinitely many points the operator $U_f$ is not compact, secondly, whenever there is no $n\in \N$ with $f^{n}(X) = f^{n+1}(X)$ then $\sigma(U_f) = \overline{B_1(0)}$ (see Theorem 2.7 respectively Theorem 3.0.2 in the literature \cite{Scheffold,KariKuester}) and if there is a non-periodic point $x \in X$ then $U_f$ is not a spectral operator\cite{fixman1959problems}.
	\item For $Y = \mathrm{L}^2(X,\mathcal{B},\mu)$ where $\mathcal{B}$ denotes the Borel sigma algebra on $X$, $f$ is assumed to be Borel measureable and essentially invertible and $\mu$ is an invariant measure, i.e. $K_f \mu = \mu$ for $K_f$ as in (\ref{equ:DefPerronFrobenius}): Then $U_f$ is unitary (in particular not compact if $X$ is not finite) but a spectral operator and $K_f = T_{f^{-1}}$ in this case, where $T_{f^{-1}}$ is the composition operator given by $T_{f^{-1}}g=g\circ f^{-1}$.
\end{enumerate}
\end{example}

\subsection{Reproducing kernel Banach spaces}
The concept follows the idea of reproducing kernel Hilbert spaces but aims to extend this concept to (pairs of) Banach spaces of functions where instead of an inner product we have a bilinear form on the (pair of) Banach spaces.\cite{RKBSUnified}
That means keeping the property of continuous point evaluation but at the same time allowing different geometries than these that arise from an inner product. Natural examples are any finite dimensional function spaces equipped with a bilinear form and any norm -- if this norm is not induced by a Hilbert space the corresponding space is not an RKHS but an RKBS.

\begin{definition}[Reproducing kernel Banach space\cite{RKBSUnified}]
    \label{def of RKBS}
    Let $X$ be a set and $\B$ be a Banach space of functions on $X$ (where the addition and multiplication with a scalar are defined pointwise). We call $\B$ an RKBS if the point evaluation $\B \ni g \mapsto g(x)$ is continuous for all $x \in X$.
\end{definition}

Similar to RKHS we want to relate the RKBS to kernels because they allow a more explicit description of the metric.

\begin{definition}[Kernels for RKBS\cite{RKBSUnified}]
    A quadrupel $(\B,\B',\langle \cdot,\cdot\rangle,k)$ is called an RKBS with kernel $k$ if $\B$ is an RKBS on a set $X$, $\B'$ a Banach space of functions on a set $Y$, $\langle \cdot,\cdot \rangle: \B \times \B' \rightarrow \mathbb{C}$ a continuous bilinear form and $k:X \times Y \rightarrow \mathbb{C}$ is such that for all $x \in X$ we have $k(x,\cdot) \in \B'$ and
\begin{equation}
    g(x) = \langle g,k(x,\cdot)\rangle \; \; \text{ for all } g \in \B.
\end{equation}
If further $\B'$ is also an RKBS and for all $y \in Y$ we have $k(\cdot,y) \in \B$ and
\begin{equation}
    h(y) = \langle k(\cdot,y),h\rangle \; \; \text{ for all } h \in \B'
\end{equation}
then we call $\B'$ an adjoint RKBS. If $Y = X$ we call $(\B,\B',\langle \cdot,\cdot\rangle,k)$ an RKBS on $X$ with kernel $k$.
\end{definition}

\begin{remark}
    If $(\B,\B',\langle \cdot,\cdot\rangle,k)$ is an RKBS with $\B = \B'$, where $\B'$ is a Hilbert space with scalar product $\langle\cdot,\cdot\rangle$, then $\B$ is an RKHS with kernel $k$. Vice versa an RKHS $\H$ with scalar product $\langle\cdot,\cdot\rangle$ with kernel $k$ induces naturally the RKBS $(\H,\H,\langle\cdot,\cdot\rangle,k)$.
\end{remark}

\begin{remark}
    The continuous bilinear form $\langle\cdot,\cdot\rangle$ induces a map
    \begin{equation}\label{eq:DefInducedMapB'B*}
        \phi:\B' \rightarrow \B^*, h \mapsto \langle \cdot,h\rangle
    \end{equation}
    where $\B^*$ denotes the dual space of $\B$. The map $\phi$ is continuous due to the continuity of the bilinear form and represents how far the bilinear form differs from the natural pairing of $\B$ and its dual $\B^*$. \new{In the case of RKHS the map $\phi$ is exactly the Fr\'echet-Riesz isomorphism between and $\H$ and its dual.}
\end{remark}

Next, we define the pullback kernel which allows to define an RKBS on a set $Y$ based on an RKBS on a set $X$ and an embedding $\psi:Y \rightarrow X$. As the name suggests, we pull back the RKBS structure via $\psi$.

\begin{lemma}[Pullback kernel]\label{PullbackKernelRKBS}
    Let $(\B,\B',\langle\cdot,\cdot\rangle,k)$ be an RKBS on $X$ with kernel $k$ and $\phi:Y \rightarrow X$ be a bijective map. Then $(\B_\phi,\B_\phi',\langle\cdot,\cdot\rangle_\phi,k_\phi)$ is an RKBS on $Y$ with kernel for
    \begin{equation}\label{equ:PullBackB}
        \B_\phi := \{g \circ \phi: g \in \B\} \text{ with norm } \|h\|_{\B_\phi} := \|h \circ \phi^{-1}\|_\B
    \end{equation}
    and
    \begin{equation}\label{equ:PullbackB'}
        \B_\phi' := \{g \circ \phi: g \in \B'\} \text{ with norm } \|h\|_{\B'_\phi} := \|h \circ \phi^{-1}\|_{\B'},
    \end{equation}
    with bilinear form
\begin{equation}\label{equ:PullbackBilinearform}
    \langle h,h'\rangle_\phi := \langle h \circ \phi^{-1},h' \circ \phi^{-1}\rangle
\end{equation}
and kernel
\begin{equation}\label{equ:PullbackKernel}
    k_\phi:Y \times Y \rightarrow \mathbb{K}, \; \;  k_\phi(y_1,y_2) := k(\phi(y_1),\phi(y_2)),
\end{equation}
where $\mathbb{K}$ denotes $\mathbb{R}$ or $\mathbb{C}$. Further the composition operator $T_\phi$ with $T_\phi g := g \circ \phi$ defines isometric isomorphisms between $\B$ and $\B_\phi$ and $\B'$ and $\B'_\phi$ and preservers the bilinear forms, i.e. $\langle T_\phi g,T_\phi h\rangle_\phi = \langle g,h\rangle$.
\end{lemma}

\begin{proof}
    By definition of $\B_\phi$ and $\B'_\phi$ it follows that $T_\phi$ induces isometric isomorphisms from $\B$ to $\B_\phi$ and from $\B'$ to $\B_\phi'$. Hence $\B_\phi$ and $\B_\phi'$ are Banach spaces (of functions on $Y$). Similarly, we see that $\langle T_\phi g,T_\phi h\rangle_\phi = \langle g,h\rangle$, and in particular $\langle\cdot,\cdot\rangle_\phi$ is continuous on $\B_\phi \times \B'_\phi$. It remains to check the reproducing property, this also follows from the (pull back) definition, namely we have for all $h = g \circ \phi \in \B_\phi$ and $y \in Y$
    \begin{align*}
        h(y) &= g(\phi(y)) = \langle g,k(\phi(y),\cdot)\rangle = \langle T_\phi g,T_\phi k(\phi(y),\cdot)\rangle_\phi\\
        &= \langle h,k_\phi(y,\cdot)\rangle. \qedhere
    \end{align*}
\end{proof}


\section{Koopman and Perron-Frobenius operators on RKBS; discrete time systems}\label{sec:Koopman_Op_RKBS}

In this section we define the Koopman and Perron-Frobenius operator on RKBS $\RKBS$ on $X$ and (discrete time) dynamics $f:X\rightarrow X$. We state general properties of these operators including fundamental properties concerning their functorial nature and address continuity by investigating closedness and boundedness of these operators.

We begin this section by motivating how the Perron-Frobenius operator acts on an RKBS. This is in analogy to the case when the Koopman operator is considered on the space of continuous functions on $X$, i.e. the Perron-Frobenius operator acts on Borel measures. Similar to RKBS, the space $\C(X)$ enjoys continuous point evaluation -- which are given by the action of the dirac measures $\delta_x$ for $x \in X$. The map $x \mapsto \delta_x \in M(X)$ provides an injective embedding of $X$ into the space of regular Borel measures on $X$, analogous to the map $x \mapsto k(x,\cdot)$, and both elements $\delta_x$ and $k(x,\cdot)$ represent the point evaluation $g \mapsto g(x)$. The Perron-Frobenius operator on $M(X)$ acts on measures by (\ref{equ:DefPerronFrobenius}), in particular it maps $\delta_x$ to $\delta_{f(x)}$. Hence it seems natural for the Perron-Frobenius operator on RKBS to send $k(x,\cdot)$ to $k(f(x),\cdot)$. And this is exactly what it does.\cite{IIS20, RosenfeldSystemIdentification}

\begin{remark}
A fundamental difference between the point evaluations in $\C(X)$ and a reproducing kernel Hilbert space is that the evaluation functionals $\delta_x$ on $\C(X)$ are extremal points of the unit ball in the dual space $\C(X)^*$. This geometric characterization of the point evaluations does not have to be true in RKBSs.
\end{remark}

\subsection{Definitions of the Koopman and Perron-Frobenius operator}

Before we state the technical requirements for defining the Koopman and Perron-Frobenius operator we want to mention the following duality in defining these two operators: The functional description $U_fg := g \circ f$ can be stated easily but it can be difficult to verify for which $g$ in the function space $\B$ it holds that $U_fg$ is still an element of $\B$. On the other hand we will see that the Perron-Frobenius operator $K_f$ can be naturally defined on the dense subset $\mathrm{Span}\{k(x,\cdot) : x\in X\}$ using $K_f k(x,\cdot) := k(f(x),\cdot)$; but at the same time a functional expression for $K_f h$ for an arbitrary element $h \in \B'$ is not obvious. In Theorem \ref{ElementaryPropertiesKoopmanRKHS} we will see that this asymmetry in the behaviour of the Koopman and Perron-Frobenius operators relates to the closablility of the Perron-Frobenius operator.

As mentioned, we want to define the linear operator $K_f$ via $K_fk(x,\cdot) := k(f(x),\cdot)$. To guarantee that the map $k(x,\cdot) \mapsto k(f(x),\cdot)$ is well defined, it is useful to assume that $k(x_1,\cdot),\ldots,k(x_n,\cdot)$ are linearly independent for any $n \in \N$ and any choice of pairwise distinct point $x_1,\ldots,x_n \in X$.

\begin{assumption}\label{Assumption:SpanLinearIndep} We assume that the set $\{k(x,\cdot) : x \in X\} \subset \B'$ is linearly independent. 
\end{assumption}

\begin{remark}
    For RKHS the $\{k(x,\cdot) : x \in X\}$ is linearly independent if and only if $k$ is a strictly positive kernel, that is, the kernel $k$ satisfies for all $n \in \mathbb{N}$, $(a_1,\ldots,a_n) \in \mathbb{C}^n \setminus \{0\}, \; (x_1,\ldots,x_n) \in X^n$
    \begin{equation*}
			\sum\limits_{i,j = 1}^n a_i\overline{a}_j k(x_i,x_j) > 0.
		\end{equation*}
\end{remark}

\new{\begin{remark}\label{rem:GuidelineRKHS}
    Most of the following concepts concerning the Koopman and Perron-Frobenius operator on RKBS follow well understood machinery for Hilbert spaces, such as RKHS. Sometimes the necessary technical parts risk hiding the underlying idea even though they are in most of the cases motivated by their analog parts for Hilbert spaces, such as the adjoint operator (with respect to a bilinear form), the notion of density from (\ref{equ:DenseB}) or the embedding $\phi$ of $\B'$ into $\B^*$ from (\ref{eq:DefInducedMapB'B*}).
\end{remark}}

\begin{definition}[Koopman and Perron-Frobenius operator]\label{def:KoopmanPerronFrobeniusRKBS}
    Let $(\B,\B',\langle \cdot,\cdot\rangle,k)$ be an RKBS with kernel such that Assumption \ref{Assumption:SpanLinearIndep} is satisfied. Let $f:X \rightarrow X$ be given dynamics. The Koopman operator $U_f:\B \supset D(U) \rightarrow \B$ is defined by
    \begin{equation}\label{def:KoopmanOperator}
        U_fg:= g \circ f \; \; \; \text{ for } g \in D(U) := \{h \in \B : h \circ f \in \B\}.
    \end{equation}
    The Perron-Frobenius operator $K_f:\SpanX \rightarrow \SpanX \subset \B'$ is defined by
\begin{equation}\label{eq:def:Kf:OnPoints}
    K_fk(x,\cdot) := k(f(x),\cdot) \text{ for } x \in X
\end{equation}
    and extended linearly to $\SpanX$.
\end{definition}

Assumption \ref{Assumption:SpanLinearIndep} guarantees that extending (\ref{eq:def:Kf:OnPoints}) linearly to $\SpanX$ is well defined. We will see that the Koopman operator is the adjoint operator of the Perron-Frobenius operator, but therefore we need the Perron-Frobenius operator to be densely defined. Therefore we choose the following notion of density\cite{RKBSUnified}; a set $W \subset \B$ respectively $W' \subset \B'$ is called dense with respect to $\langle \cdot,\cdot\rangle$ if
\begin{equation}\label{equ:DenseB}
    \langle w,g \rangle = 0 \text{ for all } w \in W \text{ implies } g = 0
\end{equation}
and analog for $W'$
\begin{equation}\label{equ:DenseB'}
    \langle v,w' \rangle = 0 \text{ for all } w' \in W' \text{ implies } v = 0.
\end{equation}
 In the case of $W=\B$ and $W' = \B'$, the conditions (\ref{equ:DenseB}) and (\ref{equ:DenseB'}) state that the dual form is non-degenerate. Condition (\ref{equ:DenseB}) is a reformulation of the map $\phi$ from (\ref{eq:DefInducedMapB'B*}) being injective and (\ref{equ:DenseB'}) states that we can embed $\B$ into $(\B')^*$. Hence the conditions (\ref{equ:DenseB}) and (\ref{equ:DenseB'}) describe foremost an algebraic property of the bilinear form $\langle \cdot,\cdot\rangle$ and therefore should not be mistaken with the notion of density with respect to the topologies on $\B$ and $\B'$.

\begin{remark} \label{Assumption:SpanDense} The set $\SpanX$ is dense in $\B'$ with respect to $\langle \cdot,\cdot\rangle$ because for any $g \in \B$ with $0 = \langle g,h\rangle$ for all $h \in \B'$, we have in particular $g(x) = \langle g,k(x,\cdot)\rangle = 0$, i.e. $g$ is the zero function. 
For reflexive RKBS, in particular RKHS, we get also that $\B$ is dense in $\B$.
\end{remark}

The following result states that the Perron-Frobenius operator is adjoint (with respect to $\langle\cdot,\cdot\rangle$) to the Koopman operator. It extends the result from the RKHS setting.\cite{Rosenfeld}
Note that we use the notation $A'$ for the adjoint with respect to a bilinear form $\langle \cdot,\cdot\rangle$ (see Section \ref{app:AdjointOperators} in the appendix) and $A^*$ for the classical adjoint operator.

\begin{lemma}\label{lem:PerronFrobeniusRKBS}
    Let $(\B,\B',\langle \cdot,\cdot\rangle,k)$ be an RKBS with kernel satisfying Assumption \ref{Assumption:SpanLinearIndep}. Then $K_f$ is densely defined with respect to $\langle\cdot,\cdot\rangle$ and we have $U_f = K_f'$ (with respect to $\langle\cdot,\cdot\rangle$).
\end{lemma}

\begin{proof}
    Since we assume that the set $\{k(x,\cdot): x \in X\}$ is linearly independent the Perron-Frobenius operator is well defined. By Remark \ref{Assumption:SpanDense} $\SpanX$ is dense in $\B'$ with respect to $\langle\cdot,\cdot\rangle$ and by Lemma \ref{Lem:AdjointClosed} the adjoint of $K_f$ exists and is unique. To check that $U_f$ is the adjoint of $K_f$ let $g \in D(K_f')$ then for all $x \in X$ we have
    \begin{align*}
        K_f'g(x) &= \langle K_f'g,k(x,\cdot)\rangle = \langle g,K_fk(x,\cdot) \rangle = \langle g,k(f(x),\cdot)\rangle\\
        &= g(f(x)) = U_fg(x).
    \end{align*}
    This shows that $U_f$ is at least an extension of $K_f'$. For $g \in D(U_f)$, i.e. $g \in \B$ such that $g \circ f \in \B$ we have
    \begin{eqnarray*}
        \langle g,K_fk(x,\cdot) \rangle & = & \langle g,k(f(x),\cdot) \rangle = g(f(x)) = (g \circ f) (x)\\
        & = &\langle g\circ f, k(x,\cdot)\rangle = \langle U_fg,k(x,\cdot)\rangle.
    \end{eqnarray*}
    Hence we have $K_f' = U_f$.
\end{proof}

\begin{remark}
    It is shown that if an operator $K$ leaves the set $\{k(x,\cdot) : x \in X\}$ invariant then $K$ is a Perron-Frobenius operator\cite{CompOpAnalyticFunctions} (see Theorem 1.4 in the literature).
\end{remark}

\subsection{Basic properties}

In Theorem \ref{ElementaryPropertiesKoopmanRKHS} we present a collection of fundamental properties of the Koopman and Perron-Frobenius operator on RKBS. Before stating it we want to put it into context with existing results for Koopman and Perron-Frobenius operators on RKHS. The first three statements in Theorem \ref{ElementaryPropertiesKoopmanRKHS} transfer from classical arguments for composition operators; in particular it shows that the information about the dynamical system is incorporated in the Koopman operator (statement 3. in Theorem \ref{ElementaryPropertiesKoopmanRKHS}). Statement 4. is an extension from existing results for the RKHS setting\cite{RosenfeldSystemIdentification,paulsen}. Statement 6. is a transfer of a classical result for adjoint operators to the RKBS setting and statement 8. relates to kernel-mean embeddings \cite{KernelMeanEmbedding}.

\begin{theorem}\label{ElementaryPropertiesKoopmanRKHS}
Let $f,\tilde{f}:X \rightarrow X$ be two maps and $(\B,\B',\langle \cdot,\cdot\rangle,k)$ be an RKBS on $X$ with kernel $k$ satisfying Assumption \ref{Assumption:SpanLinearIndep}. Then
\begin{enumerate}
    \item $K_fK_{\tilde{f}} = K_{f\circ \tilde{f}}$
    \item if $f$ is a bijection then $K_f^{-1} = K_{f^{-1}}$
    \item If $\B$ is dense in $\B$ with respect to $\langle\cdot,\cdot\rangle$ then
\begin{equation*}
	f = g \text{ if and only if } K_f = K_g.
\end{equation*}
\item $U_f$ is closed (with respect to the weak as well as norm topology). In particular, $U_f$ is bounded if and only if $D(U_f) = \B$.
\item Assume $X$ is compact and $\B$ has the universal property (see Definition \ref{Def:UnivProperty}), $f:X\rightarrow X$ is continuous and one of the following holds
\begin{enumerate}
    \item The map $\phi$ from (\ref{eq:DefInducedMapB'B*}) is an isomorphism
    \item \label{5b} $x \mapsto k(x,\cdot) \in \B'$ is continuous
\end{enumerate}
Then, if $X$ contains infinitely many elements, the operator $K_f$ is not closed with respect to $\langle\cdot,\cdot\rangle$.
\item \label{6} If $U_f$ is densely defined then $K_f$ is closeable. If the map $\phi$ from (\ref{eq:DefInducedMapB'B*}) is an isomorphism and $\B$ is reflexive then the converse is true as well, i.e. if $K_f$ is closeable then $U_f$ is densely defined.
\item Assume the map $\phi$ from (\ref{eq:DefInducedMapB'B*}) is an isomorphism.  If $D(U_f) = \B$ then $K_f$ can be extended to a bounded operator on $\B'$. If, in addition, $\B$ is reflexive then the converse is true as well.
\item Under the assumptions of point 5., the operator $K_f$ can be extended to
\begin{equation}\label{eq:DomainExtensionK}
    D:= \left\{\int\limits_Xk(x,\cdot) \; d\mu(x): \mu \in M(X)\right\}
\end{equation}
by
\begin{equation}\label{eq:ExtentionKToMeasures}
    \bar{K}_f \left( \int\limits_Xk(x,\cdot) \; d\mu(x)\right) := \int\limits_X k(f(x),\cdot) \; d\mu(x) \text{ for } \mu \in M(X).
\end{equation}
where $M(X)$ denotes the set of Borel measures on $X$.
\end{enumerate}
\end{theorem}

\begin{proof}We have for all $x \in X$
    \begin{align*}
        K_fK_{\tilde{f}}k(x,\cdot) &= K_fk(\tilde{f}(x),\cdot) = k(f(\tilde{f}(x)),\cdot) = k((f\circ \tilde{f})(x),\cdot)\\
        &= K_{f \circ \tilde{f}}k(x,\cdot).
    \end{align*}
    Hence $K_fK_{\tilde{f}} = K_{f \circ \tilde{f}}$ on $\mathrm{Span}\{k(x,\cdot) : x \in X\}$ and if follows the first statement. In particular it follows $K_f^{-1} = K_{f^{-1}}$ if $f$ is invertible. For the third statement for $f = g$ it is obvious that also $K_f = K_g$. Assume now $K_f = K_g$. Then for $x \in X$ and all $h \in \B$
    \begin{eqnarray*}
         0 & = & \langle h,(K_f - K_g)k(x,\cdot)\rangle = \langle h,k(f(x),\cdot) - k(g(x),\cdot)\rangle.
    \end{eqnarray*}
    Hence, since we assumed $\B$ to be dense in $\B$ with respect to $\langle \cdot,\cdot\rangle$, it follows $k(f(x),\cdot) = k(g(x),\cdot)$. From Assumption \ref{Assumption:SpanLinearIndep}, it follows $f(x) = g(x)$. The fourth statement follows from $U_f = K_f'$ (by Lemma \ref{lem:PerronFrobeniusRKBS}), Lemma \ref{Lem:AdjointClosed} and the closed graph theorem. We will show the fifth statement at last once we have proven 8. For 6., if $U_f$ is densely defined then $B:= U_f'$ is a closed extension of $K_f$. If $(\B,\B',\langle\cdot,\cdot\rangle)$ is reflexive then the second statement in 6. follows from Proposition \ref{prop:AdjointDenselyDefined}. 
    In the case that $D(U_f) = \B$ then by 4. we have that $U_f$ is bounded. The idea is to use the adjoint of $U_f$ together with the isomorphism $\phi$ to define a natural candidate for an extension of $K_f$. We define the bounded operator $T:= \phi^{-1}U_f^*\phi:\B'\rightarrow \B'$, where $U_f^*:\B^* \rightarrow \B^*$ denotes the (classical) adjoint of $U_f$. We claim that $T$ extends $K_f$. To check this let $x \in X$ and $g\in \B$, then by definition of $\phi$ and Lemma \ref{lem:PerronFrobeniusRKBS}
    \begin{align*}
        \langle g,Tk(x,\cdot)\rangle & = \langle g,\phi^{-1}U_f^*\phi k(x,\cdot)\rangle = (U_f^*\phi k(x,\cdot))(g)\\
        & = (\phi k(x,\cdot))(U_fg) = \langle U_fg,k(x,\cdot)\rangle\\
        & = \langle g,K_fk(x,\cdot)\rangle.
    \end{align*}
    From which it follows $Tk(x,\cdot) = K_fk(x,\cdot)$ because $\phi$ is injective (or in other words, $\B$ is dense in $\B$ with respect to $\langle\cdot,\cdot\rangle$). For the second statement of 7. we assume that $K_f$ has a bounded extension $K:\B' \rightarrow \B'$ and $\B$ reflexive and want to show that $U_f$ is bounded. The idea is very similar but the adjoint of $K$ is an operator on $\B'^* \cong \B^{**}$ and in order to find an operator on $\B$ we use that $\B$ is reflexive. That $\B$ is reflexive means that the map
    \begin{equation}\label{eq:ReflexiveMap}
        J:\B \rightarrow \B^{**}, J(b)(b^*) := b^*(b)
    \end{equation}
    is an isomorphism. We define the candidate operator
    \begin{equation}\label{eq:DefOpU}
        U:= J^{-1}(\phi^{*})^{-1}K^*\phi^*J:\B \rightarrow \B.
    \end{equation}
    The operator $U$ from (\ref{eq:DefOpU}) is bounded and we claim that $U = U_f$. To check this let $g \in \B$ and $x \in X$. Then playing with the definition of $\phi,\phi^*$ and $J$ gives
    \begin{align*}
        Ug(x) &=  \langle Ug,k(x,\cdot)\rangle = \phi(k(x,\cdot))(Ug)\\
        &= \phi(k(x,\cdot)) \left( J^{-1}(\phi^{*})^{-1}K^*\phi^*Jg\right)\\
        &= \left((\phi^{*})^{-1}K^*\phi^*Jg\right)(\phi(k(x,\cdot))\\
        &= \left(K^*\phi^*Jg\right)(\phi^{-1}\phi(k(x,\cdot)) = (K^*\phi^*Jg)k(x,\cdot)\\
        &= (\phi^*Jg)(Kk(x,\cdot)) = (\phi^*Jg)k(f(x),\cdot)\\
        &= Jg(\phi k(f(x),\cdot)) = \phi(k(f(x),\cdot))(g)\\
        &= \langle g,k(f(x),\cdot) = g(f(x)) = U_fg(x).
    \end{align*}
    To show statement 8., we separate the two cases of assumptions (a) and (b) from 5. In the case of (b) note first that the (Bochner) integrals in (\ref{eq:DomainExtensionK}) and (\ref{eq:ExtentionKToMeasures}) exist due to the continuity assumptions on $k$ and $f$. By choosing $\mu$ to be a dirac delta $\delta_y$ for some $y \in X$ we get \begin{align*}
        \bar{K}_f k(y,\cdot) &= \bar{K}_f\left( \int\limits_X k(x,\cdot) \; d\delta_y(x)\right) = \int\limits_X k(f(x),\cdot) \; d\delta_y(x)\\
        &= k(f(y),\cdot).
    \end{align*}
    That shows that $\bar{K}_f$ extends $K_f$. It remains to show that (\ref{eq:ExtentionKToMeasures}) is well defined. That means whenever there are two measures $\mu,\nu \in M(X)$ with \begin{equation}\label{eq:NonUniqueRep}
        \int\limits_Xk(x,\cdot) \; d\mu(x) = \int\limits_Xk(x,\cdot) \; d\nu(x)
    \end{equation}
    then also $\int\limits_Xk(f(x),\cdot) \; d\mu(x) = \int\limits_Xk(f(x),\cdot) \; d\nu(x)$. This follows trivially once we have shown that the representation of (\ref{eq:NonUniqueRep}) is unique, i.e. (\ref{eq:NonUniqueRep}) implies $\mu = \nu$. From (\ref{eq:NonUniqueRep}) we get for all $g \in \B$ by continuity of the bilinear form
    \begin{eqnarray*}
        \int\limits_X g(x) \; d\mu(x) & = & \int\limits_X \langle g,k(x,\cdot)\rangle \; \mu(x) =        \bigg\langle g,\int\limits_X k(x,\cdot) \; d\mu(x) \bigg\rangle\\
        & = & \bigg\langle g,\int\limits_X k(x,\cdot) \; d\nu(x)\bigg\rangle = \int\limits_X g(x) \; d\nu(x).
    \end{eqnarray*}
    The universal property together with the Riesz-Markov representation theorem implies now that $\mu = \nu$. To show that we can extend $K_f$ by (\ref{eq:ExtentionKToMeasures}), also in case of assumption (a) from 5., we first show that the argument in (\ref{eq:ExtentionKToMeasures}) as well as the proposed image have representations based on the embedding $i:\B \rightarrow \C(X)$ (more precisely its adjoint $i^*:M(X) \rightarrow \B^*$), the isomorphism $\phi :\B^* \rightarrow \B'$ with $b^*(b) = \langle b,\phi(b^*)\rangle$ for all $b^* \in \B^*$ and $b\in \B$, and the Perron-Frobenius operator $P_f$ on $M(X)$ from (\ref{equ:DefPerronFrobenius}). Note that here the term $\int\limits_X k(x,\cdot) \; d\mu(x)$ is understood in the weak sense, that is, for each $g \in \B$ we have
    \begin{equation}\label{eq:weakdefintKx}
        \langle g,\int\limits_X k(x,\cdot) \; d\mu(x) \rangle := \int\limits_X \langle g,k(x,\cdot)\rangle \; d\mu(x) = \int\limits_X g(x) \; d\mu(x).
    \end{equation}
    Next, we claim that $\int\limits_X k(x,\cdot) \; d\mu(x)$, is nothing else than $\phi(i^*\mu)$ for all $\mu \in M(X)$. This can be seen as follows: For any $g \in \B$ we have
    \begin{eqnarray*}
        \langle g,\phi(i^* \mu)\rangle & = & (i^* \mu)(g) = \int\limits_X i(g)(x) \; d\mu(x) = \int\limits_X g(x) \; d\mu(x)\\
        & \overset{(\ref{eq:weakdefintKx})}{=} & \langle g,\int\limits_X k(x,\cdot) \; d\mu(x) \rangle.
    \end{eqnarray*}
    Similarly for the right-hand side of (\ref{eq:ExtentionKToMeasures}). Namely, for any $g \in \B$    \begin{eqnarray*}
        \bigg\langle g, \int\limits_X k(f(x),\cdot) \; d\mu(x)\bigg\rangle & = & \int\limits_X g(f(x)) \; d\mu(x) = \int\limits_X g \; dP_f\mu\\
        & = & (i^*(P_f \mu))(g) = \langle g,\phi(i^*(P_f \mu))\rangle.
    \end{eqnarray*}
    where $P_f$ denotes the Perron-Frobenius operator from (\ref{eq:KoopmanPerronFrobAdjoint}). That means (\ref{eq:ExtentionKToMeasures}) states that we want to extend $K_f$ to the range of $\phi \circ i^*$, i.e. $D$, by setting
    \begin{equation}\label{eq:ExtensionByiAndPhi}
        \bar{K}_f (\phi(i^* \mu)) := \phi(i^*P_f\mu)).
    \end{equation}First let us check that this is well defined. By the universal property, $i^*$ is injective (Remark \ref{rem:i*Injective}) and hence $\phi \circ i^*$ is injective, too -- hence (\ref{eq:ExtensionByiAndPhi}) is well defined. Finally, to see that $\bar{K}_f$ is indeed an extension of $K_f$ we have show that $\bar{K}_fk(x,\cdot) = k(f(x),\cdot)$ for all $x \in X$. As in the previous case we use that $\phi(i^*\delta_x) = k(x,\cdot)$ for any $x \in X$, from which it follows
    \begin{eqnarray*}
        \bar{K}_f k(x,\cdot) & = & \bar{K}_f (\phi(i^*\delta_x)) = \phi(i^*(P_f\delta_x)) = \phi(i^* \delta_{f(x)})\\
        & = & k(f(x),\cdot).
    \end{eqnarray*}
    This shows 8. under the assumption (b) from 5.
    Last, it remains to show 5. The property that is important in this proof is that weak* convergence of measures $\mu_n$ to $\mu \in M(X)$, denoted by $\mu_n \overset{*}{\rightharpoonup} \mu$, implies
    \begin{equation}\label{eq:weak*convBilinearFormConv}
     \bigg\langle g,\int\limits_X k(x,\cdot) \; d\mu_n(x)\bigg\rangle \rightarrow \langle g,\int\limits_X k(x,\cdot) \; d\mu(x)\rangle 
    \end{equation}
    for all $g \in \B$. This follows directly from the weak* convergence of $\mu_n$, namely
    \begin{align*}
        \bigg\langle g,\int\limits_X k(x,\cdot) \; d\mu_n(x)\bigg\rangle &= \int\limits_X g \; d\mu_n\rightarrow \int\limits_X g \; d\mu\\
        &= \langle g,\int\limits_X k(x,\cdot) \; d\mu(x)\rangle.
    \end{align*}
     We use the extension $\bar{K}_f$ of $K_f$ from 8. and show $\bar{K}_f = K_f$ if $K_f$ was closed with respect to $\langle \cdot,\cdot\rangle$. But this will lead to a contradiction because we will see that the domain of $\bar{K}_f$ is strictly greater than the domain of $K_f$. Let $\mu \in M(X)$. We may assume that $\mu$ represents a non-negative measure -- otherwise, apply the Hahn-Jordan decomposition to $\mu$. By scaling we may assume that $\mu$ is a probability measure. Then for $n \in \N$ there exist $x_1^{(n)},\ldots,x_{k_n}^{(n)} \in X$ and $\lambda_1^{(n)},\ldots,\lambda_{k_n}^{(n)} \geq 0$ with $\sum\limits_{i = 1}^{k_n} \lambda_i^{(n)} = 1$ such that
\begin{equation}\label{equ:deltaxConvergeToMu}
    \mu_n := \sum\limits_{i = 1}^{k_n} \lambda_i^{(n)} \delta_{x_i^{(n)}} \overset{*}{\rightharpoonup} \mu \; \; \text{ as } n \rightarrow \infty.
\end{equation}
    By continuity of the Perron-Frobenius operator $P_f$ on $M(X)$ from (\ref{equ:DefPerronFrobenius}) we then also have
    \begin{equation}\label{eq:weakcontPerronFrob}
        \sum\limits_{i = 1}^{k_n} \lambda_i^{(n)} \delta_{f(x_i^{(n)})} = P_f\mu_n \overset{*}{\rightharpoonup} P_f \mu
    \end{equation}
    For for any $g \in \B$ we get from (\ref{equ:deltaxConvergeToMu}) 
    \begin{eqnarray*}
        \langle g, \sum\limits_{i = 1}^{k_n} \lambda_i^{(n)} k(x_i^{(n)},\cdot) \rangle & = & \bigg\langle g, \int\limits_X k(x,\cdot) \; d\mu_n(x) \bigg\rangle\\
        & \rightarrow & \bigg\langle g,\int\limits_X k(x,\cdot)\; d\mu(x)\bigg \rangle
    \end{eqnarray*}
    and from (\ref{eq:weakcontPerronFrob})
    \begin{eqnarray*}
        \langle g, K_f \sum\limits_{i = 1}^{k_n} \lambda_i^{(n)} k(x_i^{(n)},\cdot) \rangle & = &  \langle g, \sum\limits_{i = 1}^{k_n} \lambda_i^{(n)} k(f(x_i^{(n)}),\cdot) \rangle \\
        & = & \bigg\langle g, \int\limits_X k(f(x),\cdot) \; d\mu_n(x)\bigg\rangle\\
        & = & \bigg\langle g, \int\limits_X k(x,\cdot) \; dP_f\mu_n(x)\bigg\rangle \\
        & \rightarrow & \bigg\langle g, \int\limits_X k(x,\cdot) \; dP_f\mu(x)\bigg\rangle
    \end{eqnarray*}
    Because we assumed that $K_f$ was closed with respect to $\langle\cdot,\cdot\rangle$ it follows in particular that $\int\limits_X k(x,\cdot)\; d\mu(x) \in D(K_f) = \SpanX$. That means we can find $m \in \N$, $y_1,\ldots,y_m \in X$, $a_1,\ldots,a_m \in \R$ with
    \begin{equation}
        \int\limits_X k(x,\cdot)\; d\mu(x) = \sum\limits_{i = 1}^m a_i k(y_i,\cdot) \text{ in } \B',
    \end{equation}
    which means for all $g \in \B$ we have
    \begin{equation*}
        \int\limits_X g \; d\mu = \sum\limits_{i = 1}^m a_i g(y_i) =  \int\limits_X g \; d\left( \sum\limits_{i = 1}^m a_i \delta_{y_i}\right).
    \end{equation*}
    From the universal property it follows $\mu = \sum\limits_{i = 1}^m a_i \delta_{y_i}$, i.e. $\mu$ is atomic. Since $\mu$ was arbitrary that means all Borel measures $\mu \in M(X)$ are atomic -- which contradicts Lemma \ref{AppendixLemmaLinearCombination} since $X$ contains infinitely many points.
\end{proof}

\begin{remark}
     In many cases in Theorem \ref{ElementaryPropertiesKoopmanRKHS} we could not overcome the need for regularity that we are used to from Hilbert spaces, such as reflexivity and the isomorphism between $\H$ and it dual space (here this role is played by $\phi$). That means for RKHS the regularity assumptions in Theorem \ref{ElementaryPropertiesKoopmanRKHS}) are redundant. We can partially overcome the regularity assumptions on $\RKBS$ by imposing regularity on the map $x \mapsto k(x,\cdot)$ for instance, as in \ref{5b}. Another possibility is in \ref{6} where instead of reflexivity of $\B$ it is possible to derive the same statement under the condition that $K_f$ has a bounded extension $K$ such that $\phi K \phi^{-1}$ is weak* continuous (from which it follows that there exists a bounded operator $U:\B \rightarrow \B$ with $U^* = \phi K \phi^{-1}$ and we can argue similarly as in the proof to show that $U = U_f$). 
\end{remark}

\begin{remark}[Invariant kernels]\label{rem:InvariantKernel}
    An easy (but restrictive) setting that guarantees boundedness of the operator $U_f$ on an RKHS $\H$ with kernel $k$ is invariance of $k$, i.e. for all $x,y \in X$
    \begin{equation}\label{eq:InvKernelDef}
        k(f(x),f(y)) = k(x,y).
    \end{equation}
    In this case $U_f$ and $K_f$ are isometries, due to
    \begin{eqnarray*}
        \left\|K_f\sum\limits_{i = 1}^n a_ik(x_i,\cdot)\right\|^2 & = & \left\|\sum\limits_{i = 1}^n a_ik(f(x_i),\cdot)\right\|^2 \\
        & = & \sum\limits_{i,j = 1}^n a_i \overline{a}_j k(f(x_i),f(x_j))\\
        & = & \sum\limits_{i,j = 1}^n a_i\overline{a}_j k(x_i,x_j) = \left\|k(x,\cdot)\right\|^2
    \end{eqnarray*}
    for all $n \in \N$ and $a_1,\ldots,a_n \in \mathbb{C}$. More generally, by the same arguments, the Perron-Frobenius operator is bounded with $\|K_t\| \leq M$ if and only if we have
    \begin{equation}\label{eq:KfBoundedExplicit}
        \sum\limits_{i,j = 1}^n a_i \overline{a}_j k(f(x_i),f(x_j)) \leq M\sum\limits_{i,j = 1}^n a_i\overline{a}_j k(x_i,x_j).
    \end{equation}
    In contrast to (\ref{eq:InvKernelDef}) the condition (\ref{eq:KfBoundedExplicit}) is typically not easily verified.
\end{remark}

Concerning continuity and domain, it is clear that the treatment of $U_f$ on an RKBS is more subtle than working on $\C(X)$ for continuous dynamics or $\mathrm{L}^2(X,\mu)$ for measure preserving dynamics (where the Koopman and Perron-Frobenius operators are bounded) -- for an RKBS it can happen that the condition $g \circ f \in \B$ might not be satisfied for any function $g \in \B \setminus \{0\}$. This indicates that the RKBS (and the kernel $k$) need to be chosen corresponding to the function $f$. This is a very natural condition because we want the RKBS to capture information about $f$.

One possibility of defining an RKBS such that the Koopman operators are bounded uses conjugacy and follows the classical concept for dynamical systems that sometimes (local) charts give better insight into the dynamics.

\begin{proposition}\label{LemmaConjugacy}
    Let $f:X \rightarrow X$ and $(\B,\B',\langle\cdot,\cdot\rangle,k)$ be an RKBS on $X$ with kernel. Let $g:Y \rightarrow Y$ such that there exists a bijective function $\phi: Y \rightarrow X$ with $\phi \circ g = f \circ \phi$. Let  $(\B_\phi,B'_\phi,\langle\cdot,\cdot\rangle_\phi,k_\phi)$ be the corresponding pullback RKBS with kernel from Lemma \ref{PullbackKernelRKBS}. Then 
    \begin{equation}\label{IntertwiningConjugacyRKBS}
        K_f T_\phi = T_\phi K_g.
    \end{equation}
    In particular if $K_f$ is bounded on $\B'$ then so is $K_g$ with $\|\overline{K}_g\| = \| \overline{K}_f\|$.
\end{proposition}

\begin{proof} By Lemma \ref{PullbackKernelRKBS} we have that $T_\phi$ is an isometric isomorphism. Hence it remains to show (\ref{IntertwiningConjugacyRKBS}). For any $y \in Y$ we have
    \begin{align*}
        K_fT_\phi k_\phi(y,\cdot) & = K_f k(\phi(y),\cdot) = k(f(\phi(y)),\cdot)\\
                                  & = k(\phi(g(y)),\cdot) = T_\phi k_\phi(g(y),\cdot)\\
                                  & = T_\phi K_g k_\phi (y,\cdot). \qedhere
    \end{align*}
\end{proof}

Proposition \ref{LemmaConjugacy} can be exploited
when we have the knowledge of a suited RKBS for a conjugated system.

\subsection{Examples}

In this section, we present several examples from the literature. Example \ref{example:RKHSRn} is of introductory nature and covers linear (or finite) dynamics and Example \ref{ex:CXRKBS} recovers the case of the Koopman operator acting $\C(X)$ from Example \ref{example:KoopmanCXL2} from an RKBS perspective. Other examples treat holomorphic dynamics (Example \ref{ex:Hardyspace}), point out limitations of the approach (Examples \ref{ex:PositiveDefiniteFunction} and \ref{ex:ShiftInvariantKernel}), focus on polynomial dynamics and provide situations in which the domain of the Koopman operator contains the set of all polynomials or where a connection to well-posedness of a transport equation is drawn via boundedness of Koopman operator.

Those examples, particularly the limiting ones, demonstrate that not any RKBS fits the dynamical system at hand, and properties of the dynamical system, such as linearity or regularity, have to be considered for the choice of the kernel.

For our first example, we view $\R^n$ as an RKBS.

\begin{example}[$\R^n$ as an RKBS and linear systems]\label{example:RKHSRn}
There are two very natural ways to interpret $\R^n$ as an RKBS. The first is that $\R^n$ is interpreted as the space of functions from $X := \{1,\ldots,n\}$ to $\R$, i.e. we identify $(x_1,\ldots,x_n) \in \R^n$ with the map $x(\cdot):X \rightarrow \R$ given by $x(i):= x_i$ for $i = 1,\ldots,n$. Since that space is finite dimensional it is a Banach space for any norm we choose and any linear operator is bounded, including the point evaluation $x(\cdot) \mapsto x(i)$ for $i = 1,\ldots,n$. The second way to view $\R^n$ as an RKBS is to interpret $\R^n$ as the dual space of $\R^n$, that is we view an element $x \in \R^n$ as a linear map from $\R^n$ to $\R$. This can be done by fixing a bilinear form $\langle\cdot,\cdot\rangle$ on $\R^n$; then each $x \in \R^n$ induces the linear map $\hat{x}:\R^n \rightarrow \R$ given by $\hat{x}(a) := \langle a,x\rangle$. Again, due to finite dimensions, the point evaluation for $\hat{x}$ is continuous. In the following we will make the above constructions more precise and induce corresponding kernels as well.

The first case, i.e. viewing $\R^n$ as the space of real valued functions on $X = \{1,\ldots,n\}$, is well suited for dynamical systems on the discrete set $X$. We denote the dynamics on $X$ with $i \mapsto \sigma(i)$ for $i \in X= \{1,\ldots,n\}$. We set $\B = \B' = \{h:X\rightarrow \R\}$ and we identify each $h \in \B$ with a vector $\bar{h} := (h(1),\ldots,h(n)) \in \R^n$. For the bilinear form we choose $\langle h,g\rangle := \bar{h}^T M \bar{g}$ for $h,g \in \B = \B'$ and an invertible matrix $M \in \R^{n \times n}$ respectively $M \in \mathbb{C}^{n \times n}$. Invertibility of $M$ on the one hand assures that $\B$ is dense in $\B$ with respect to $\langle\cdot,\cdot\rangle$. On the other hand for all $i = 1,\ldots,n$
\begin{equation}\label{eq:DefKRn}
    \bar{k}(\cdot,i) = M^{-1}e_i, \; \bar{k}(i,\cdot) = (M^T)^{-1}e_i
\end{equation}
where $e_i := (\delta_{ij})_{j = 1}^n \in \R^n$, which shows that $\{k(x,\cdot): x \in X\}$ is linearly dependent. Equation (\ref{eq:DefKRn}) follows from $\langle h,k(\cdot,i)\rangle = h(i) = \bar{h}^T e_i = \bar{h}^TM M^{-1}e_i$. In particular, $\B'$ is an adjoint RKBS. Clearly the Koopman operator $U_\sigma$ is well defined and acts by
\begin{equation}
    U_\sigma h = h \circ \sigma
\end{equation}
and
\begin{equation*}
    \overline{U_\sigma h} = P_\sigma \bar{h}
\end{equation*}
where $P_\sigma$ denotes the permutation matrix with $(P_\sigma)_{ij} = 1$ if $\sigma(i) = j$ and $(P_\sigma)_{ij} = 0$ otherwise. The Perron-Frobenius operator has the form
\begin{equation}\label{eq:PerrFrobDiscreteRn}
    K_\sigma k(i,\cdot) = k(\sigma(i),\cdot)
\end{equation}
and hence
\begin{equation*}
    \overline{K_\sigma k(i,\cdot)} = M^{-1} e_{\sigma(i)}.
\end{equation*}
Hence the representation matrix of $K_\sigma$ with respect to the basis $e_1,\ldots,e_n$ of $\R^n$ (respectively $\mathbb{C}^n$) is given by $M^{-1}P_\sigma M$. If $M$ is chosen so that it diagonalizes $P_\sigma$ then the Koopman operator has a diagonal representation with respect to the standard basis of $e_1,\ldots,e_n$.

For the second case we choose $X = \R^n \cong \B = \B'$ where we identify an element $a \in \R^n$ with a map $g_a:X \rightarrow \R$ by $g_a(x) := a^T x$, i.e. $\B$ and $\B'$ consist of linear forms on $\R^n$. Again we choose the bilinear form $\langle g_a,g_b\rangle := a^T M b$ for an invertible matrix $M \in \R^{n \times n}$. Similar to the discrete case we have for $x \in \R^n$ that $k(x,\cdot) = g_{M^{-1}x}$ because
\begin{equation*}
    \langle g_a,g_{M^{-1}x}\rangle = a^TMM^{-1}x = a^Tx = g_a(x).
\end{equation*}
From a dimension argument it is clear that $\{k(x,\cdot): x \in X\}$ is not linearly independent so we do not know a-priori that the Perron-Frobenius operator defined as in Definition \ref{def:KoopmanPerronFrobeniusRKBS} is well defined. Let the dynamics be given by $x_{k+1} = Ax_k$ for a matrix $A \in \R^{n\times n}$. The Koopman operator is given by
\begin{equation}
    U_A g_a(x) = g_a(Ax) = a^T Ax = (A^Ta)^Tx = g_{A^Ta}(x)
\end{equation}
for all $x$, i.e. $U_A$ has matrix representation $A^T$ with respect to the basis $g_{e_1},\ldots,g_{e_n}$, and
\begin{equation}\label{eq:PerronFrobLinearForm}
    K_Ak(x,\cdot) = k(Ax,\cdot) = g_{M^{-1}Ax}.
\end{equation}
And we see that the basis representation of $K_A$ in the standard basis $e_1,\ldots,e_n$ is given by $MAM^{-1}$.
\end{example}

Next, we turn our attention to the classical example of the Koopman operator on $\C(X)$ for continuous dynamics $f$ from Example \ref{example:KoopmanCXL2}. We reformulate it as a Koopman operator on the RKBS $\C(X)$.

\begin{example}[$\C(X)$ as an RKBS.]\label{ex:CXRKBS}
    For compact $X$, we view $\C(X)$ equipped with the supremum norm $\|\cdot\|_\infty$ as an RKBS with kernel.\cite{RKBSUnified} Clearly $\C(X)$ enjoys bounded point evaluation, hence $\B := \C(X)$ is an RKBS. But there is freedom in the choice of $\B'$ and the kernel\cite{RKBSUnified}. We follow the construction from\cite{RKBSUnified}. Let $X$ be compact and $k:X \times X \rightarrow \R$ continuous such that $\mathrm{Span}\{k(\cdot,x): x \in X\}$ is a dense subset of $\C(X)$. For $X = [0,1]$ examples of such $k$ are $k(x,y) = 1- |x-y|$, $k(x,y) = e^{xy}$ and $k(x,y) = (1+y)^x$. We define the RKBS in the following way: Let $\B = \C(X)$ and $\B'$ be the space of kernel mean embeddings, i.e.
    \begin{equation}\label{eq:ExRKBSCXKernelEbedding}
        \B' = \left\{g_\mu : \mu \in M(X), g_\mu(x):= \int\limits_X k(y,x) \; d\mu(y)\right\}
    \end{equation}
    and the bilinear form $\langle\cdot,\cdot\rangle:\B \times \B' \rightarrow \R$ is given by
    \begin{equation}\label{eq:BilinearFormCX}
        \langle h,g_\mu\rangle := \int\limits_X h \; d\mu.
    \end{equation}
    The condition that $\SpanX$ is dense in $\C(X)$ guarantees that the bilinear form (\ref{eq:BilinearFormCX}) is well defined. Then $\RKBS$ is an RKBS with kernel $k$.\cite{RKBSUnified}
    To verify that $k$ is a kernel let $x\in X$. Because for $y \in X$ we have $k(x,\cdot)(y) = k(x,y) = \int\limits_X k(x,z)\; d\delta_y(z)$, i.e. $k(x,\cdot) = g_{\delta_x}$ (with the notion from (\ref{eq:ExRKBSCXKernelEbedding})). It follows for all $h \in \C(X)$
    \begin{equation*}
        \langle h,k(x,\cdot)\rangle = \langle h,g_{\delta_x}\rangle = \int\limits_X h \; d\delta_x = h(x).
    \end{equation*}
    Further $k$ is an adjoint kernel as well. To check this let $\mu \in M(X)$. For $g_\mu$ we have
    \begin{equation*}
        g_\mu(x) = \int\limits_X k(y,x) \; d\mu(y) = \int\limits_X k(\cdot,x) \; d\mu = \langle k(\cdot,x),g_\mu\rangle.
    \end{equation*}
    For the Koopman operator, we get that $U_f$ has domain $\B = \C(X)$ because $g \circ f$ is continuous whenever $g$ is (since we assume $f$ to be continuous). Hence $K_f$ can be extended to a bounded operator on $\B'$ by Theorem \ref{ElementaryPropertiesKoopmanRKHS}). Note that $\C(X)$ is not unique as an RKBS with kernel. Similarly, $K_f$ depends on the kernel $k$. For the examples $k(x,y) = 1- |x-y|$, $k(x,y) = e^{xy}$ and $k(x,y) = (1+y)^x$ for $X = [0,1]$ we get
    \begin{equation}
        K_f: \begin{cases} 1- |x-\cdot| \mapsto 1-|f(x)-\cdot|\\
                           e^{x\cdot} \mapsto e^{f(x)\cdot}\\
                           (1+\cdot)^x \mapsto (1+\cdot)^{f(x)}.
        \end{cases}
    \end{equation}
    
\end{example}

In the next example, we consider Hardy spaces $H^p(D)$ where $p \geq 1$ and $D$ is the unit disc $D \subset \mathbb{C}$. The Hardy space $H^p(D)$ consists of all analytic functions on $D$ for which the following norm is finite
\begin{equation}
    \|g\|_{H^p} := \sup\limits_{0\leq r < 1} \left(\int\limits_0^{2\pi} | f(re^{i\theta})|^p \; d\theta\right)^{\frac{1}{p}}.
\end{equation}

\begin{example}\label{ex:Hardyspace}
The kernel for the Hardy space $H^p(D)$ is given by the Szegö kernel $k(z,w) := \frac{1}{1-z\bar{w}}$ and turns $B:= H^p(D)$ into an RKBS where we take the dual-pairing $\langle\cdot,\cdot\rangle$ of $H^p(D)$ and its dual space and we set $\B' := \overline{\mathrm{Span}\{k(z,\cdot) : z \in D\}}$ where the closure is taken in the topological dual space of $H^p(D)$. By \cite{CompOpAnalyticFunctions} (Theorem 3.6 in the literature) a holomorphic automorphism $f:D \rightarrow D$ has a bounded Koopman operator on $H^p(D)$ with $\|U_f\|^p = \frac{1+| f(0)|}{1-| f(0) |}$. For dynamics given by a Möbiustransform $f(z):= \lambda \frac{z-a}{1-z \bar{a}}$ for $a,\lambda \in \mathbb{C}$ with $| \lambda | = 1$ and $|a| < 1$ we define the map $\phi(z) := \frac{z-\gamma}{1-z \bar{\gamma}}$ with the unique fixed point $\gamma \in D$ of $f$. It can be shown that the pull-back kernel $k_\phi$ is an invariant kernel for $f$ and thus $U_{f}$ is an isometry (see Remark \ref{rem:InvariantKernel}) on the pullback RKBS (see Lemma \ref{PullbackKernelRKBS}) and $K_{f}$ and can be extended to an isometry on the same space as well. In \cite{russo2022liouville} the authors go further and consider weighted composition operators on the Hardy space and show stronger boundedness results for this setting.
\end{example}

The following example presents a class of RKHS that does not allow for compact Koopman operators.
\new{
\begin{example}\label{ex:PositiveDefiniteFunction}
Let $G$ be a locally compact abelian group and let $\widehat{G}$ be the character group of $G$, the set of continuous group homomorphisms from $G$ to the unit circle in $\mathbb{C}$.
Let ${\rm d}g$ and ${\rm d}\chi$ be Haar measures of $G$ and $\widehat{G}$, respectively.
    A map $u:G\rightarrow{\mathbb C}$ is a {\em positive definite function} if $k(x,y):=u(x-y)$ is a positive definite kernel. 
We call $\H$ the RKHS associated with $u$.
Thanks to Bochner's theorem\cite{Katznelson-text-04}, a positive definite function on $G$ can be realized as a Fourier transform of a finite Borel measure. 
Namely, in the case where $u$ is continuous, $u$ is a positive definite function if and only if there exists a finite Borel measure $\mu$ on $\widehat{G}$ such that 
\[u(x)
=\widehat{\mu}(x)
:=\int_{\widehat{G}}\chi(x)\,d\mu(\chi).\]
Let us consider the case $\mu = w(\chi){\rm d}\chi$ where ${\rm d}\chi$ is the Haar measure of $\widehat{G}$ and $w \in L^1 \cap L^\infty \setminus \{0\}$ and $w\ge0$ almost everywhere. 
Then, if $k(x,y) = \widehat{\mu}(x-y)$, we have
\[ \H = \left\{h\in C^0\cap L^2(G): \widehat{h}\in L^p(\widehat{G}, w^{-1})\right\}.\]
Here, we define the Fourier transform $\widehat{h} \in L^2(\widehat{G})$ of $h \in L^2(G)$ by
\[\widehat{h}(\chi):= \int_G h(g)\overline{\chi(g)} {\rm d}g.\]
It is known that under certain conditions on $w$, no composition operator is compact on $\H$\cite{IIS20} in the case of $G = \mathbb{R}^d$.
We note that when $G=\mathbb{R}^d$, we usually regard $\widehat{G}$ as $\mathbb{R}^d$ via the correspondence
\[\mathbb{R}^d \ni x \mapsto [\xi \mapsto {\rm e}^{2\pi i x \cdot \xi}]. \]
The general description of the Fourier transform on the locally compact abelian group defined above is equivalent to the usual Fourier transform via the above correspondence when $G = \mathbb{R}^d$
\end{example}
}
We continue the previous example by considering shift invariant kernels, which include the popular setting of Gaussian kernel RKHS.

\begin{example}[Shift invariant kernels]\label{ex:ShiftInvariantKernel}
    A kernel $k$ on $\R^n$ (or any group) is called shift invariant if for all $x,y,a \in \R^n$ we have $k(x+a,y+a) = k(x,y)$. So shift invariant kernel are the kernels that are invariant kernels for all translation maps $f_a:\R^n 
    \rightarrow \R^n$ with $f_a(x) = x+a$.\\
    Kernels of the form
    \begin{equation}\label{equ:ExampleKernelInducedByH}
        k(x,y) = h(\|x-y\|)
    \end{equation}
    for some positive definite function $h$ are typical examples for shift invariant kernels. A function $h$ is called positive definite if the corresponding kernel (\ref{equ:ExampleKernelInducedByH}) is positive definite.
    For example the Gauss kernel with parameter $\sigma > 0$ given by
    \begin{equation}
        k(x,y) = \frac{1}{\sigma \sqrt{2\pi}} e^{-\frac{\|x-y\|^2}{2 \sigma^2}}
    \end{equation}
    is positive definite and the corresponding RKHS is dense in the space of continuous functions on $\R^n$ that vanish at infinity \cite{sriperumbudur2011universality}. For the Gaussian kernel RKHS the only dynamics that induce bounded Koopman and Perron-Frobenius operators are affine ones.\cite{IsaoBoundedCompOp} Nevertheless, Koopman analysis in this setting has been successfully applied to forecasting\cite{Kawahara16,alexander2020operator} and system identification\cite{RosenfeldSystemIdentification}
\end{example}

The next example treats some RKBS and RKHS that (densely) contain polynomials. This is of particular importance in case of polynomial dynamics $f$, because then the Koopman operator on those spaces is well-defined at least on the set of polynomials.

\begin{example}\label{ex:Polynomials}
    The easiest example of an RKBS containing polynomials is the space $\R[x]_d$ of polynomials up to a fixed degree $d \in \N$. Because this space is finite dimensional it can be made an RKBS and a kernel can be chosen as $k(x,y) := (1+x^Ty)^d$. This space is Koopman invariant only if the dynamics $f$ are affine. For polynomial dynamics $f$ of degree $\deg (f) = k$ we have $D(U_f) \supset \R[x]_{d-k}$. Among examples of RKBS which contain all polynomials are $\C(X)$ from Example \ref{ex:CXRKBS}, the Bragmann-Fock space\cite{rosenfeld2022dynamic} with kernel $k(x,y) := e^{\bar{x}^Ty}$ consisting of the holomorphic functions $g$ on $\mathbb{C}^n$ with finite integrals $\int\limits_{\mathbb{C}^n} g(z) e^{-\|z\|^2} \; dz$, the Bergman space $A^p(G)$ of $p$-integrable holomorphic functions on a domain $G \subset \mathbb{C}^n$ (with kernel $k(z,w) = \frac{1}{(1-\bar{z}w)^2}$ in the case where $G \subset \mathbb{C}$ is the unit disc) and Sobolev spaces (see Example \ref{ex:Sobolev}). \commentCS{In\cite{mezic2020spectrum}, the Koopman operator on the Bragmann-Fock space was investigated and boundedness of the Koopman operator was shown for a pullback kernel obtained via principal eigenfunctions. We refer to\cite{mezic2020spectrum} for more examples of RKHS appearing naturally for spectral expansions of the Koopman operators $U_f$.} In the above-mentioned examples the set of polynomials is even dense with respect to the corresponding topologies; on the contrary, the RKHSs corresponding to the Gaussian kernel and the Cauchy kernel don't contain any non-zero polynomial\cite{dette2021reproducing}.
\end{example}

The following example is important in signal processing, time-warping\cite{azizi1999preservation} and sampling\cite{bergner2006spectral} and is another example of an RKHS that contains the space of polynomials.

\begin{example}\label{ex:BandLimitedFunctions}
    We consider the space of band-limited functions  $PW_A := \{\hat{g} : g \in \L^2((-A,A))\}$ where $\hat{g}$ denotes the Fourier-transform of a function $g$ and $A \in (0,\infty)$ is the band width. The space $PW_A$ is an RKHS\cite{paulsen, SS16} with kernel $k(x,y) = \mathrm{sinc}(A(x-y)) := \frac{1}{\pi} \frac{\sin (2\pi A (x-y))}{x-y}$ for $x \neq y$ and $k(x,x):= 2A$. The Koopman operator appears in time warping, that is when for the incoming signal $f(x)$ only the signal $h(f(x)) = U_fh (x)$ is observed\cite{azizi1999preservation}. The bandwidth of the time-warped signal $h(f(x))$ determines the Nyquist sampling rate\cite{bergner2006spectral} and therefore the question of whether $h$ is in the domain of $U_f$, i.e. whether $U_fh$ still belongs to $PW_A$ is important. It was shown in \cite{mukherjee2011range} that only injective affine maps induce bounded Koopman operators on the space of band-limited functions on $\R^n$, and the same result holds for the larger function space $\cup_{A > 0}PW_A$\cite{lebedev2012functions}. The celebrated Payley-Wiener Theorem\cite{SS16} characterizes the functions in $PW_A$ by their exponential growth -- in particular, this shows that the space of polynomials is contained in $PW_A$.
\end{example}

Another class of examples arises from Sobolev spaces with enough regularity and we state an easy condition for which diffeomorphisms induce bounded Koopman and Perron-Frobenius operators on these spaces.

\begin{example}[Sobolev space]\label{ex:Sobolev} 
    For $\Omega \subset \R^n$ open and bounded with $C^1$ boundary. For $s \in \N$ and $p \in [2,\infty)$ we denote by $W^{s,p}(\Omega)$ the Sobolev space\cite{Brezis} of functions with $p$-integrable weak derivatives up to order $s$. If $s > \frac{n}{p}$ the Sobolev embedding\cite{Brezis} tells that $\mathrm{W}^{s,p}(\Omega)$ is a subspace of $\C(\overline{\Omega})$ and there is a constant $C$ with $\| g\|_\infty \leq C \|g\|_{W^{s,p}}$. In such cases for $q \in [1,\infty)$ with $\frac{1}{p} + \frac{1}{q} = 1$ we can turn $\B := \mathrm{W}^{s,p}(\Omega)$ into an RKBS with the universal property. To reduce notation we restrict to the one-dimensional case, i.e. $\Omega \subset \R$. The higher dimensional situation is analogous. We set $\B':= \mathrm{W}^{s,q}(\Omega)$ and the bilinear form 
    \begin{equation}\label{eq:bilinearformSobolevspace}
       \langle g,h\rangle := \sum\limits_{j = 0}^s \int\limits_{\Omega} g^{(j)}(x) h^{(j)}(x) \; dx
    \end{equation}
    where $g^{(j)}$ respectively $h^{(j)}$ denotes the $j$-th weak derivate of $g$ respectively $h$. The existence of a kernel $k$ follows because $W^{s,2}(\Omega)$ is an RKHS\cite{paulsen} and hence admits a kernel $k$ -- the same $k$ provides a kernel for $(\B,\B',\langle \cdot,\cdot\rangle)$. For $n = s = 1$ and $\Omega = (0,1)$ the kernel is given by\cite{paulsen}
    \begin{equation}\label{eq:SobolevKernel}
        k(x,y) = \begin{cases} (1-y)x, & x \leq y\\
                              (1-x)y, & x \geq y.
                 \end{cases}
    \end{equation}
    We can define the Perron-Frobenius operator but in this case, it is easier to check that the Koopman operator is bounded. By Theorem \ref{ElementaryPropertiesKoopmanRKHS} the boundedness of the Koopman operator is equivalent to $D(U_f) = \B$. We show that for $\Omega = (0,1) \subset \R$\new{, $s = 1$} and diffeomorphic  $f:(0,1)\rightarrow (0,1)$ such that $f'$ and $\frac{1}{f'}$ are bounded we indeed have $D(U_f) = \B$. For $g \in \B \cap \mathcal{C}^1(U)$ we get
    \begin{eqnarray*}
        \|U_fg\|_{\mathrm{L}^p}^p & = & \int\limits_0^1 g(f(x))^p \; dx\\
        & = & \int\limits_{f(0)}^{f(1)} g(y)^p \frac{1}{f'(f^{-1}(y))} \; dy \leq \|\frac{1}{f'}\|_\infty \|g\|_{\mathrm{L}^p}^p
    \end{eqnarray*}
    and from $(U_fg)' = (g \circ f)' = g' \circ f \cdot f'$ we get
    \begin{eqnarray*}
        \|(U_fg)'\|_{\mathrm{L}^p}^p & = & \int\limits_0^1 g'(f(x))^p f'(x)^p \; dx\\
        & = & \int\limits_{f(0)}^{f(1)} g'(y)^p f'(f^{-1}(y))^{p-1} \; dy \leq \|f'\|^{p-1}_\infty \|g'\|_{\mathrm{L}^p}^p.
    \end{eqnarray*}
    This shows that $U_f\big|_{\B \cap \mathcal{C}^1(U)}:\H \cap \mathcal{C}^1(U) \rightarrow \B$ is a bounded operator. Since $\B \cap \C^1(U)$ is dense in $\B$ we can uniquely extend $U_f\big|_{\B \cap \mathcal{C}^1(U)}$ to a bounded operator $T$ on $\B$. \new{ It remains to check that $T$ is nothing else then $U_f$. For that we use that $U_f$ is closed by Theorem \ref{ElementaryPropertiesKoopmanRKHS} and that $\C^1(U) \cap \B$ is dense in $\B$. Let $g \in \B$ and $g_m \in \C^1(U)\cap \B$ with $g_m \rightarrow g$ in $\B$ as $m \rightarrow \infty$. Because the operator $T$ is bounded we get $U_f g_m = Tg_m \rightarrow Tg$. From $U_f$ being closed it follows $g \in D(U_f)$ and $U_fg = T_g$, i.e. $T = U_f$. } Hence $U_f$ is bounded and it follows from Theorem \ref{ElementaryPropertiesKoopmanRKHS} 7. that $\Kf = U_f^*$ is bounded, too. \new{For $s > 1$ similar arguments show boundedness of $U_f$ on the RKBS $(\B := W^{s,p}(0,1),\B' := W^{s,q}(0,1),\langle\cdot,\cdot\rangle,k)$ for $\langle\cdot,\cdot\rangle$ from (\ref{eq:bilinearformSobolevspace}) and $k$ from (\ref{eq:SobolevKernel}). In this case, we extend the assumption that $\frac{1}{f'}$ and $f'$ are bounded to the assumption that $\frac{1}{f'}$ and all derivatives of $f$ up to order $s$ are bounded.} We refer to \cite{menovschikov2021composition} for detailed investigations of composition operators on Sobolev spaces.
   
\end{example}

Extending the previous example concerning Sobolev spaces we provide next a theoretically crucial example, the Besov space $B^{s}_{p,q}(\mathbb{R}^n)$ for $s>n/p$ and $p,q\in (0,\infty]$.
\begin{example}\label{ex:BesovSpace}
The Besov space $B^s_{p,q}$ plays a crucial role in partial differential equations and harmonic analysis and coincides with the Sobolev space $W^{s,p}(\mathbb{R}^n)$ when $p=q$.
We emphasize that the RKBS covers the Besov space as a special example.
Here, we provide a definition of the Besov space $B^s_{p,q}(\mathbb{R}^n)$ in the case where $0 < p,q\le \infty$ and $s > \max (0, 1/p -1)$. We always assume $s>n/p$ and $p,q\in [1,\infty]$, which implies $B^s_{p,q}(\mathbb{R}^n)\subset \C^0(\mathbb{R}^n)$. There are generalized definitions for any $s,p,q$ on a domain\cite{Triebelbook,Taniguchi19}. The Besov space $B^s_{p,q}(\mathbb{R}^n)$ is the collection of measurable functions $f$ on $\mathbb R$ such that 
\[
\|f\|_{B^s_{p,q}} 
:= 
\|f\|_{L^p}
+
\left(
\int_{|h|\le 1}
|h|^{-sq}
\|\Delta_h^mf\|_{L^p}^q
\frac{dh}{|h|}
\right)^\frac{1}{q} < \infty
\]
(with the usual modification for $q=\infty$), 
where $m\in \mathbb N$ with $m>s$.
Here, the difference operator $\Delta^m_h$ of order $m\in \mathbb N$ is defined by 
\[
\Delta^m_hf(x) := \sum_{j=0}^m(-1)^{m-j}
\begin{pmatrix}
m\\
j 
\end{pmatrix}
f(x + j h),\quad x, h\in \mathbb R
\]
and $\Delta^0_h$ is the identity operator.

As we mention above, the Besov space $B^{s}_{p,q}(\mathbb{R}^n)$ for $s>n/p$ is (continuously) included in the space of continuous functions, and thus the evaluation map is obviously continuous on $B^s_{p,q}$.
Therefore, the Besov space is an RKBS (Definition \ref{def of RKBS}).
The Koopman operators on the Besov space is promising machinery for the analysis of the transport equation.
The solution of the transport equation
\[
\begin{cases}
	\partial_t u (t,x) - b(t,x) \partial_x u(t,x) = 0,\\
	u(0,x) = u_0(x)
\end{cases}
\]
can be represented by the Koopman operator \cite{Xia-2019}:
\[u(t,x) = u_0 ( \chi_t (x) ) = K_{\chi_t} u_0(x),\]
where  $\chi_t(x) = \chi (t,x)$ is the solution of
\[
	\partial_t \chi(t,x) = b(t, \chi(t,x)),\quad 
	\chi(0,x) =x.
\]
Thus, the boundedness of Koopman operators implies continuous dependence of initial values\cite{Xia-2019}.
\end{example}

\section{Continuous time systems}\label{sec:ContinuousTime}

In this section, we treat continuous time systems. They are important for many applications, especially in engineering, where the system obeys an autonomous ordinary differential equation of the form $\dot{x} = f(x)$ for a vector field $f$. Often discrete time systems arise as discretizations of continuous time systems but even in that case, there are fundamental differences between those two cases. While for discrete time systems the evolution $x_k \mapsto f(x_{k+1})$ is explicit, for continuous time systems typically only the infinitesimal evolution $f$, i.e. $\dot{x} = f(x)$, is known.

In this section, we define the Koopman and Perron-Frobenius semigroups -- similar to the discrete case. We describe the infinitesimal generator of both semigroups and relate it to the vector field $f$. Finally, we give a geometric condition under which the Koopman and Perron-Frobenius semigroup are strongly continuous and consist of bounded operators.

\subsection{Koopman and Perron-Frobenius semigroup}

The definition of the Koopman and Perron-Frobenius semigroup is based on a semiflow $\varphi_t$, see (\ref{eq:SemiflowProp}), but whenever we investigate their generators we assume that the dynamical system is induced by an ordinary differential equation $\dot{x} = f(x)$ for a vector field $f$ on $X$. In that cases, we implicitly assume that $X$ is either a (compact) smooth manifold (with boundary) and $f$ a smooth vector field on $X$ or $X$ is a subset of $\R^n$ and $f:\R^n \rightarrow \R^n$ a locally Lipschitz continuous map.

For a given differential equation $\dot{x} = f(x)$ and $x(0) = x_0$ we denote its solution map by $\varphi$, i.e. $\varphi_t(x_0)$ denotes a solution at time time $t$ to the initial value $x_0$. If solutions to the differential equation are unique, the map $\varphi$ satisfies the following semiflow property
\begin{equation}\label{eq:SemiflowProp}
    \varphi_0(x) = x \text{ and } \varphi_{t+s}(x) = \varphi_t(\varphi_s(x)) \; \text{for all } t,s \in \R_+
\end{equation}
for $x \in X$ such that $\varphi_t(\varphi_s(x))$ and $\varphi_{t+s}(x)$ are elements of $X$.

\begin{assumption}\label{Assumption:DynSysContTime}
    We assume that the set $X$ is positively invariant with respect to the differential equation $\dot{x} = f(x)$, i.e. $\varphi_t(x) \in X$ for all $x \in X$ and $t \in \R_+$, and that $\varphi$ satisfies (\ref{eq:SemiflowProp}).
\end{assumption}

\begin{definition}[Koopman semigroup]
    Let $Y$ be a (normed) space of functions on $X$. Under Assumption \ref{Assumption:DynSysContTime} the Koopman semigroup associated to a semiflow $\varphi$ is a family of operators $(U_t)_{t \in \R_+}$, for each $t \in \R_+$, defined by
    \begin{equation}\label{eq:KoopmanContinuous}
        U_t:D(U_t) \rightarrow Y,\ \ U_tg := g \circ \varphi_t
    \end{equation}
    with $D(U_t) := \{g \in Y: g \circ \varphi_t \in Y\}$.
\end{definition}

\begin{remark}
    The Koopman operators $U_t$ are linear for all $t \in \R_+$ and satisfy the semigroup property, that is, for all $t,s \in \R_+$
    \begin{equation}\label{eq:Semigroup}
        U_tU_sg = U_{t+s} g
    \end{equation}
    for all $g \in D(U_s)$ such that $U_sg \in D(U_t)$ and $U_0 = \mathrm{Id}$.
\end{remark}

For discrete time systems the operator $U_f$ contains already all information about the dynamical system, i.e. the evolution $U_{f^n} = U_f^n$ for $n \in \N$. For continuous time systems this role is played be the infinitesimal generator.

\begin{definition}[Generator]
    The generator $G:D(G) \rightarrow Y$ of a semigroup $(T_t)_{t \in \R_+}$ on a normed space $Y$ is defined by
    \begin{equation}\label{eq:Generator}
        Gg := \lim\limits_{t \rightarrow 0} \frac{T_t g - g}{t}
    \end{equation}
    for $g \in D(G) := \{ g \in \bigcap\limits_{t \in \R_+} D(T_t): \lim\limits_{t \searrow 0} \frac{T_t g - g}{t} \text{ exists}\}$.
\end{definition}

The generator describes differentiability of the map $t \mapsto T_tg$ for $g \in D(G)$. A more fundamental property is continuity of this map. A semigroup $(T_t)_{t \in \R_+}$ on a space $V$ is called strongly continuous if $T_tv \rightarrow T_0v = v$ as $t \rightarrow 0$ for all $v \in V$.

\begin{remark}\label{rem:Perron-FrobSemigroupStronglyContonSpanx}
    If $\varphi$ and the map $x \mapsto k(x,\cdot)$ are continuous then the Perron-Frobenius semigroup $(K_t)_{t \in \R_+}$ is strongly continuous on $\SpanX$. This follows because for any $n \in \N$, $a_1,\ldots,a_n \in \R$ respectively $\mathbb{C}$ and $x_1,\ldots,x_n$ we have
    \begin{eqnarray*}
        \|(K_t - \mathrm{Id}) \left(\sum\limits_{i = 1}^n a_i k(x_i,\cdot) \right)\| & \leq & \sum\limits_{i = 1}^n |a_i| \|k(\varphi_t(x_i),\cdot) - k(x_i,\cdot)\|
    \end{eqnarray*}
    which converges to $0$ as $t \rightarrow 0$, due to continuity of $\varphi$ and $x \mapsto k(x,\cdot)$.
\end{remark}

For strongly continuous semigroups of bounded operators, the generator characterizes the semigroup uniquely.\cite{EngelNagel} Hence it allows to investigate the semigroup by investigating a single object. Strongly continuous semigroups $(T_t)_{t\in \R_+}$ of bounded operators satisfy $\frac{\mathrm{d}}{\mathrm{dt}}T_t = GT_t$ (in a certain sense\cite{EngelNagel}). So it is not surprising that there is a close connection between the ``generator" of the dynamics $f$ and the generator $A$ of the Koopman semigroup. When point evaluations are continuous then for each $g \in D(A)\cap \C^1(X)$ and $x \in X$ we get
\begin{eqnarray}\label{eq:GeneratorDirectionalDerivative}
    Ag(x) & = & \lim\limits_{t \rightarrow 0} \frac{U_tg(x) - g(x)}{t} = \lim\limits_{t \rightarrow 0} \frac{g(\varphi_t(x)) - g(x)}{t}\notag\\
    & = & \frac{\mathrm{d}}{\mathrm{dt}}\big|_{t = 0} g(\varphi_t(x)) = Dg(x) \frac{\mathrm{d}}{\mathrm{dt}}\big|_{t = 0} \varphi_t(x)\\
    & = & Dg(x) f(x).\notag
\end{eqnarray}

\begin{remark}\label{rem:GeneratorClosed}
    In the case that $Y$ is an RKHS consisting of continuously differentiable functions the generator $A$ from (\ref{eq:GeneratorDirectionalDerivative}) is a closed operator\cite{RosenfeldSystemIdentification} (see Theorem 4.2 in the literature).
\end{remark}

\subsection{Koopman and Perron-Frobenius semigroups on RKBS}

In the continuous time case, we have a family of Koopman respectively Perron-Frobenius operators -- the so-called Koopman semigroup respectively Perron-Frobenius semigroup. In contrast to the discrete time case, we want to investigate the infinitesimal generator in addition to the semigroup. First, we state some elements on which the infinitesimal generator of the Perron-Frobenius semigroup acts.\cite{RosenfeldSystemIdentification} Under additional regularity assumptions on the kernel, namely being $\C^1$, we show and use that the generator of the Perron-Frobenius semigroup acts on the kernel functions $k(x,\cdot)$ as well.

\begin{proposition}\label{prop:RKBSKoopmanSemigroup}
    Let $\RKBS$ be an RKBS with kernel $k$ and $(U_t)_{t\in \R_+}$ be the Koopman semigroup for a dynamical system with semiflow $\varphi_t$. Let $\{k(x,\cdot): x \in X\}$ be linearly independent. Then we define the Perron-Frobenius semigroup of linear operators $(K_t)_{t \in \R_+}$ with $K_t:\SpanX \rightarrow \SpanX$ for $t \in \R_+$ by linearly extending
    \begin{equation}
        K_tk(x,\cdot) = k(\varphi_t(x),\cdot).
    \end{equation}
    Further $K_t' = U_t$ for $t \in \R_+$.
\end{proposition}

\begin{proof}
    For each $t$ the operator $U_t$ coincides with the composition operator $U_{\varphi_t}$ from Section \ref{sec:Koopman_Op_RKBS}. So the result follows from Definition \ref{def:KoopmanPerronFrobeniusRKBS} and Lemma \ref{lem:DenselyDefAdjoint}.
\end{proof}

\begin{remark}
    As mentioned in the proof of Proposition \ref{prop:RKBSKoopmanSemigroup} for each $t \in \R_+$ the operator $U_t$ coincides with $U_{\varphi_t}$ from Section \ref{sec:Koopman_Op_RKBS}. Similarly for $K_t = K_{\varphi_t}$. In particular Theorem \ref{ElementaryPropertiesKoopmanRKHS} holds for $U_t$ and $K_t$ for each $t \in \R_+$.
\end{remark}

In the following we investigate the infinitesimal generator $C$ of $(K_t)_{t \in \R_+}$, i.e.
\begin{equation}\label{eq:DefGeneratorK}
    Cg := \lim\limits_{t \rightarrow 0} \frac{1}{t} \left(K_tg - g\right)
\end{equation}
whenever the limit exists. In the RKHS case,  it was presented that certain path-integrals are elements on which the infinitesimal generator $C$ of $(K_t)_{t \in \R_+}$ acts naturally\cite{RosenfeldSystemIdentification}. 
Those path-integrals are defined in the following Definition \ref{Def:IntegralFunctional}.

\begin{definition}\label{Def:IntegralFunctional}
    Let $T > 0$ and $t\mapsto k(\varphi_t(x),\cdot) \in \B'$ be continuous on $[0,T]$. For $x \in X$ and $I_{T,x} \in \B^*$ be defined by
    \begin{equation}
        I_{T,x} g := \int\limits_0^T g(\varphi_t(x)) \; dt.
    \end{equation}
    We can identify $I_{T,x}$ with the element in $\B'$ given by
    \begin{equation}\label{eq:Defb'TX}
        b'_{T,x} := \int\limits_0^T k(\varphi_t(x),\cdot) \; dt.
    \end{equation}
\end{definition}

\begin{remark}
    The continuity assumption in Definition \ref{Def:IntegralFunctional} is used to guarantee that the (Riemann) integral (\ref{eq:Defb'TX}) exists. In order to weaken the regularity on the feature map $x \mapsto k(x,\cdot)$, Bochner's theorem on Bochner integrals can be evoked -- in our case of RKBS, this would typically require more regularity of the space $\B$ and $\B'$, such as reflexivity for example.
\end{remark}

\begin{remark}
    If $x \mapsto k(x,\cdot) \in \B'$ is continuous, $X \subset \R^n$ and $f$ is locally Lipschitz then the flow map $\varphi$ is continuous, and hence $t \mapsto k(\varphi_t(x),\cdot)$ is continuous.
\end{remark}

\begin{assumption}\label{Ass:Continuity}
    We assume that $\varphi$ and $x \mapsto k(x,\cdot) \in \B'$ are continuous.
\end{assumption}

We extend $K_t$ to $I_{T,x}$ as in Theorem \ref{ElementaryPropertiesKoopmanRKHS} 8. by
\begin{equation}\label{eq:KonIntegral}
    K_t I_{T,x} = \int\limits_{0}^T k(\varphi_{t+s}(x),\cdot) \; ds.
\end{equation}
To guarantee that (\ref{eq:KonIntegral}) is well defined we will use the universal property.

\begin{assumption}\label{Ass:UniversalProp}
    We assume that $X$ is compact and $\B$ satisfies the universal property from Definition \ref{Def:UnivProperty}, i.e. that $\B$ is dense in $\C(X)$.
\end{assumption}

\begin{lemma}\label{lem:ITxDomainKt}
    Under Assumptions \ref{Assumption:SpanLinearIndep}, \ref{Assumption:DynSysContTime}, \ref{Ass:Continuity} and \ref{Ass:UniversalProp} the term $K_tI_{T,x}$ from (\ref{eq:KonIntegral}) is well defined.
\end{lemma}

\begin{proof}
    We want to argue as in Theorem \ref{ElementaryPropertiesKoopmanRKHS} 8. Therefore, it suffices to note that $I_{T,x} = \int\limits_X k(y,\cdot) \; d\mu(y)$ for the measure $\mu$ given the action $\int\limits_X g \; d\mu = \int\limits_0^T g(\varphi_t(x)) \; dt$ for $g \in \C(X)$. The result follows from Theorem \ref{ElementaryPropertiesKoopmanRKHS} 8.
\end{proof}


\begin{remark}
    If $x \in X$ is a periodic point, i.e. there exists $P > 0$ with $\varphi_P(x) = x$, then $K_tI_{P,x} = I_{P,x}$ for all $t \in \R_+$. The discrete analog of this result holds for discrete systems.
\end{remark}

We can show that $I_{T,x}$ is contained in the domain of the generator of the Perron-Frobenius semigroup.\cite{Rosenfeld}

\begin{proposition}
    Under Assumptions \ref{Assumption:SpanLinearIndep}--4, for $T > 0$ and $x \in X$, we have $I_{T,x} \in D(C)$, where $C$ denotes the generator of the Perron-Frobenius semigroup. In particular, $C$ is densely defined.
\end{proposition}

\begin{proof}
    By Lemma \ref{lem:ITxDomainKt} the element $I_{T,x}$ is contained in the domain of $K_t$ for all $t \in \R_+$. For the generator of the Perron-Frobenius semigroup, we get as $t \rightarrow 0$
    \begin{eqnarray*}
        \frac{1}{t}\left(K_t I_{T,x} - I_{T,x}\right) & = & \frac{1}{t}\left( \int\limits_0^T k(\varphi_{t+s}(x),\cdot) - k(\varphi_{s}(x),\cdot) \; ds \right)\\
        & = & \frac{1}{t}\left( \int\limits_T^{T+t} k(\varphi_t(x),\cdot) \; ds - \int\limits_0^t k(\varphi_t(x),\cdot) \; ds \right)\\
        & \rightarrow & k(\varphi_T(x),\cdot) - k(x,\cdot),
    \end{eqnarray*}
    i.e. $I_{T,x} \in D(C)$ and $CI_{T,x} = k(\varphi_T(x),\cdot) - k(x,\cdot)$. To check that $C$ is densely defined let $x \in X$. Since $D(C)$ is a linear subspace we have $\frac{1}{T} I_{T,x} \in D(C)$ for all $T  > 0$ and by definition of $I_{T,x}$ we get $k(x,\cdot) = \lim\limits_{T \rightarrow 0} \frac{1}{T} I_{T,x} \in \overline{D(C)}$. Hence $C$ is densely defined.
\end{proof}

\begin{remark}
    To treat the Perron-Frobenius semigroup from a semigroup perspective we consider the closures of $K_t$ and $C$ (if they exist). Closedness or closeability of the operator $C$ is less accessible due to the explicit definition of $K_t$ only on the elements $k(x,\cdot)$. In particular generator theorems such as the Hille-Yosida theorem\cite{EngelNagel} provide a less accessible approach. On the contrary closedness of the generator of the Koopman semigroup is known for RKHS by Remark \ref{rem:GeneratorClosed} but -- as typical for the Koopman operator -- we lack a-priori information about the domain of the generator.
\end{remark}

For Proposition \ref{prop:LumerPhillipsDirectApplication} we restrict to RKHS $\H$ and fix the following preliminaries. Let $X \subset \R^n$ be open and let the kernel $k$ be $\C^1$. Let $\partial_{x_i}k$ denote the derivative of $k$ with respect to the first variable in direction of the $i$-th standard basis vector $e_i$. Then we have\cite{SS16} $\partial_{x_i} k(x,\cdot) \in \H$ for all $i = 1,\ldots,n$ and fixed $x$, and further
\begin{equation}\label{eq:convergencePartialDerivative}
    \partial_{x_i} k(x,\cdot) = \lim\limits_{h \rightarrow 0} \frac{1}{h}(k(x+he_i,\cdot) - k(x,\cdot))
\end{equation}
converges in $\H$ (see the literature in Theorem 2.5) -- in other words the feature map $x \mapsto k(x,\cdot) \in \H$ is $\C^1$. In particular we get that for fixed $x$ the map $t \mapsto K_tk(x,\cdot) = k(\varphi_t(x),\cdot)$ is continuously differentiable if $\varphi$ is the flow map for $\dot{y} = f(y)$ for a locally Lipschitz continuous vector field $f = (f_1,\ldots,f_n)$. Hence we get that $k(x,\cdot) \in D(C)$ for the generator $C$ of the Perron-Frobenius semigroup and
\begin{align}\label{eq:PerronFrobGeneratork(x)}
    Ck(x,\cdot) &= \frac{\mathrm{d}}{\mathrm{dt}}\bigg|_{t = 0} K_tk(x,\cdot) = \frac{\mathrm{d}}{\mathrm{dt}}\bigg|_{t = 0} k(\varphi_t(x),\cdot)\notag\\
    &= \sum\limits_{l = 1}^n \partial_{x_l} k(x,\cdot) f_l(x) \in \SpanX \subset \H.
\end{align}
On the other hand, for the generator $A$ of the Koopman semigroup it is not clear whether $k(x,\cdot)$ is in the domain of $A$. If so, it acts on $k(x,\cdot)$ by $Ak(x,\cdot) = \sum\limits_{i = 1}^n \partial_{y_i} k(x,\cdot) f_i(\cdot)$, where $\partial_{y_i}$ denotes the derivative of $k$ with respect to the second variable in the direction of $e_i$. Hence it is not a-priori clear whether $Ak(x,\cdot)$ is an element of $\H$. Therefore, in Proposition \ref{prop:LumerPhillipsDirectApplication} we argue via the Perron-Frobenius semigroup.

\begin{proposition}\label{prop:LumerPhillipsDirectApplication}
    Assume Assumption \ref{Assumption:SpanLinearIndep} and Assumption \ref{Assumption:DynSysContTime}. Let $X \subset \R^n$ be open and $\H$ be an RKHS on $X$ with kernel $k \in \C^1(X\times X)$. Then for $\omega > 0$ the following are equivalent
    \begin{enumerate}
        \item \label{enum:KoopmansemigroupBounded} The Koopman semigroup is a strongly continuous semigroup with $\|U_t\| \leq e^{\omega t}$ for all $t \in \R_+$
        \item \label{enum:PerronFrobsemigroupBounded} The Perron-Frobenius semigroup can be extended to a strongly continuous semigroup $(\bar{K}_t)_{t \in \R_+}$ of bounded operators on $\H$ with $\|\bar{K}_t\| \leq e^{\omega t}$ for all $t \in \R_+$
        \item \label{enum:GeneratorDissipative} For all $n \in \N$, $a_1,\ldots,a_n \in \R$ respectively $\mathbb{C}$ and $x_1,\ldots,x_n \in X$ we have
        \vspace{-3mm}
    \begin{equation}\label{eq:LumerPhillipsDissipative}
       \hspace{8.1mm} \mathrm{Re} \sum\limits_{i,j,l} a_i \overline{a}_j f_l(x_i) \partial_{x_l} k(x_i,x_j) \leq \omega \sum\limits_{i,j = 1}^n a_i \overline{a}_j k(x_i,x_j).
    \end{equation}
    \end{enumerate}
\end{proposition}

\begin{proof}
Since $U_t = K_t^*$ for all $t \in \R_+$, i.e. the Koopman semigroup is the adjoint semigroup of the Perron-Frobenius semigroup, the strong continuity of one semigroup implies the strong continuity of the other\cite{EngelNagel} (page 9 in the literature) because $\H$ is reflexive. In the rest of the proof, we show that \ref{enum:PerronFrobsemigroupBounded}. and \ref{enum:GeneratorDissipative}. are equivalent. The essential observation is that for $g \in \SpanX$ we have by (\ref{eq:PerronFrobGeneratork(x)}) for all $t \in \R_+$
\begin{eqnarray}
    \frac{\mathrm{d}}{\mathrm{dt}} \|K_tg\|^2 & = &  \frac{\mathrm{d}}{\mathrm{dt}} \langle K_tg, K_tg\rangle = \langle \frac{\mathrm{d}}{\mathrm{dt}} K_tg,K_tg\rangle + \langle K_tg,\frac{\mathrm{d}}{\mathrm{dt}}K_tg\rangle\notag\\
    & = & \langle CK_tg,K_t\rangle + \langle K_tg,CK_tg\rangle\\
    & = & 2 \mathrm{Re}\langle CK_tg,K_t\rangle\notag.
\end{eqnarray}
Representing $g$ as $g:= \sum\limits_{i = 1}^n a_i k(x_i,\cdot))$ and evaluating in $t = 0$ gives
\begin{eqnarray}\label{eq:derivNormKt^2}
       \frac{\mathrm{d}}{\mathrm{dt}} \|K_tg\|^2\big|_{t = 0} & = & 2\mathrm{Re} \langle \sum\limits_{i = 1}^n a_i \sum\limits_{l = 1}^n \partial_{x_l} k(x_i,\cdot) f_l(x_i), \sum\limits_{j = 1}^n a_jk(x_j,\cdot)\rangle \notag\\
       & = & 2\mathrm{Re} \sum\limits_{i,j,l} a_i \overline{a}_j f_l(x_i) \partial_{x_l} k(x_i,x_j).
\end{eqnarray}
Equation (\ref{eq:derivNormKt^2}) is the central object connecting \ref{enum:PerronFrobsemigroupBounded}. and \ref{enum:GeneratorDissipative}. We begin by showing that \ref{enum:GeneratorDissipative} implies \ref{enum:PerronFrobsemigroupBounded}. To do so, we first show that $\|K_tg\| \leq e^{\omega t} \|g\|$ for all $g \in \SpanX$. Condition (\ref{eq:LumerPhillipsDissipative}) implies for $g = \sum\limits_{i = 1}^n a_i k(x_i,\cdot)$ by (\ref{eq:derivNormKt^2})
\begin{equation}\label{eq:derivKtgomega}
    \frac{\mathrm{d}}{\mathrm{dt}} \|K_tg\|^2\big|_{t = 0} \leq 2 \omega \sum\limits_{i,j = 1}^n a_i\overline{a}_j k(x_i,x_j) = 2 \omega \|g\|^2.
\end{equation}
Because $g \in \SpanX$ was arbitrary in (\ref{eq:derivKtgomega}), $\SpanX$ is $K_t$ invariant for all $t \in \R_+$ and $(K_t)_{t \in \R_+}$ is a semigroup we get for the map $u:\R_+ \rightarrow \R_+, u(t) := \|K_tg\|^2$ that
\begin{eqnarray*}
    \dot{u}(t) & = & \frac{\mathrm{d}}{\mathrm{dt}} \|K_tg\|^2 =\frac{\mathrm{d}}{\mathrm{ds}} \|K_sK_tg\|^2\big|_{s = 0} \leq 2 \omega \|K_tg\|^2 = 2 \omega u(t).
\end{eqnarray*}
By Gronwall's lemma it follows $u(t) \leq e^{2 \omega t} u(0) = e^{2 \omega t} \|g\|^2$, i.e. $\|K_tg\| = \sqrt{u(t)} \leq e^{\omega t} \|g\|$ for all $g \in \SpanX$. That shows $\|K_t\|\leq e^{\omega t}$. Further, $K_t$ is strongly continuous on $\SpanX$ by Remark \ref{rem:Perron-FrobSemigroupStronglyContonSpanx}. The continuity conditions in Remark \ref{rem:Perron-FrobSemigroupStronglyContonSpanx} are satisfied due to $k$ being continuous and $f$ being locally Lipschitz continuous.  Because $\SpanX$ is dense in $\H$, it follows then that $(K_t)_{t\in \R_+}$ can be extended to a strongly continuous semigroup on $\H$\cite{EngelNagel} with the desired growth bound (see Proposition 1.3 in the literature). For the remaining implication, \ref{enum:PerronFrobsemigroupBounded}. implies \ref{enum:GeneratorDissipative}., we argue similarly. From $\|K_t\| \leq e^{\omega t}$ we get
\begin{equation}\label{eq:AuxiliaryFunction}
 \|K_tg\|^2 \leq e^{2\omega t} \|g\|^2
\end{equation}
for all $g \in \SpanX$. Evaluating (\ref{eq:AuxiliaryFunction}) in $t = 0$ we see that both sides are equal. For the derivative with respect to $t$ in $t = 0$ this implies
\begin{equation}\label{eq:derivativeInequality}
    \frac{\mathrm{d}}{\mathrm{dt}} \|K_tg\|^2\big|_{t = 0} \leq \frac{\mathrm{d}}{\mathrm{dt}} e^{2\omega t} \|g\|^2\big|_{t = 0} = 2 \omega \|g\|^2.
\end{equation}
Choosing $g = \sum\limits_{i = 1}^n a_i k(x_i,\cdot)$ the inequality (\ref{eq:derivativeInequality}) coincides with (\ref{eq:LumerPhillipsDissipative}) by (\ref{eq:derivNormKt^2}). 
\end{proof}

\begin{remark}\label{rem:LumerPhillipsRKBSProblem}
    In Proposition \ref{prop:LumerPhillipsDirectApplication}, we made strong use of the explicit computation of
    \begin{equation}\label{eq:NormKx-Ky}
    \|\sum\limits_{i = 1}^n a_ik(x_i,\cdot)\|^2 = \sum\limits_{i = 1}^n a_i \overline{a}_j k(x_i,x_j).
    \end{equation}
    Such an explicit expression of (\ref{eq:NormKx-Ky}) is not available in RKBS in general. Hence a similar result on RKBS $\RKBS$ would require further explicit knowledge of expressing the norm in $\B'$ by $k$.
\end{remark}

The condition (\ref{eq:LumerPhillipsDissipative}) is a geometric condition that connects the dynamics $f$ with the kernel $k$. For $\omega = 0$ and $n = 1$ and real RKHS $\H$ it resembles a Lyapunov condition and at first only states that $k(\varphi_t(x),x)$ is decreasing in time. Due to the symmetry of $k$, it follows for all $x \in X$,
\begin{equation*}
    \frac{\mathrm{d}}{\mathrm{dt}}\bigg|_{t = 0} k(\varphi_t(x),\varphi_t(x)) = \nabla_x k(x,x) f(x) + \nabla_x k(x,x) f(x) \leq 0
\end{equation*}
which means that $V(x) := k(x,x) = \|k(x,\cdot)\|^2$ is a Lyapunov function. The condition (\ref{eq:LumerPhillipsDissipative}) for $\omega = 0$ extends this concept to the full RKHS because it states that $\hat{V}(g) := \frac{1}{2}\|g\|^2$ is decaying in time since for all $t \in \R_+$, we have $\frac{\mathrm{d}}{\mathrm{dt}} \hat{V}(K_tg) =\mathrm{Re} \langle CK_tg,K_tg\rangle \leq 0$, that is just a reformulation of the Perron-Frobenius semigroup being contractive on $\H$.

\section{Symmetry and sparsity patterns}\label{sec:SymmetrySparsity}

Symmetry and sparsity of dynamical systems are useful concepts to gain further insight into the evolution of a dynamical system. They describe certain invariants of the dynamical system. Knowledge of such invariants allows to deduce properties of the dynamical systems of related objects of interest, as for example attractors and invariant sets.\cite{SparsePaper, Perron-FrobSymmetry}
Furthermore, knowledge of symmetry and sparsity can and should be exploited in computations\cite{Symmetry, SparseKoopman, Perron-FrobSymmetry} for the Koopman and Perron-Frobenius operator. Particularly when the task at hand shows large computational complexity, reduction techniques are useful to reduce running time or memory limitations.

Compared to working with data directly and searching for symmetries or sparse structures in the data, we assume here the a-priori knowledge of those patterns for a dynamical system. For computations that allow to incorporate the knowledge of these patterns directly into the method (if possible) without any loss of accuracy.

In this section, we describe how the symmetry concept\cite{Symmetry} for Koopman operators carry over to RKBS. Similarly for the concept of factor systems\cite{OperatorTheoreticAspectsOfErgodicTheory} which is the notion of sparsity that we work with.

\begin{definition}[Symmetry]
    A map $\psi:X \rightarrow X$ is called a symmetry for the (discrete) dynamics induced by $f:X \rightarrow X$ if $\psi \circ f = f\circ \psi$.
\end{definition}

\begin{remark}
    Typically, the map $\psi$ is assumed to be invertible. In that case, we have 
$\psi^{-1} \circ f \circ \psi = f$.
For continuous time systems, symmetry means that $\psi \circ \varphi_t = \varphi_t \circ \psi$ for all $t \in \R_+$. In other words, $\psi$ maps solutions of the dynamical system to, again, solutions of the dynamical system. If the continuous time dynamical system is induced by the differential equation $\dot{x} = f(x)$ then an invertible smooth map $\psi:X \rightarrow X$ is a symmetry if $f = D\psi^{-1} \circ \psi \cdot f\circ \psi$. 

\end{remark}

The next proposition states that symmetries induce a commutation relation between the Koopman and Perron-Frobenius operators and their corresponding operators induced by the symmetry map.

\begin{proposition}\label{prop:Symmetry}
	Let $f:X \rightarrow X$ be the (discrete) dynamics, $\RKBS$ be an RKBS on $X$ with kernel $k$ and $\psi$ a symmetry for $f$. Let $U_\psi$ and $K_\psi$ be the Koopman and Perron-Frobenius operator with respect to $\psi$ on the RKBS. Then, the relation
	\begin{equation}\label{eq:SymmetryKoopman}
	    U_fU_\psi = U_\psi U_f
	\end{equation}
	holds on the set
	\begin{equation*}
	    \{g \in \B: g \in D(U_f) \cap D(U_\psi), U_fg \in D(U_\psi), U_\psi g \in D(U_f)\}
	\end{equation*}
	and
	\begin{equation}\label{eq:CommutingPerronFrob}
	    K_\psi K_f = K_f K_\psi \;  \text{ on } \SpanX.   
	\end{equation}
\end{proposition}
\begin{proof}
This follows directly from the definition of symmetry. We only show it for the Perron-Frobenius operator. For $x \in X$ we have
\begin{align*}
K_\psi K_f k(x,\cdot) & = & K_\psi k(f(x),\cdot) = k(\psi(f(x)),\cdot) = k(f(\psi(x)),\cdot)\\
& = & K_fk(\psi(x),\cdot) = K_fK_\psi k(x,\cdot).  \hspace{20mm}\qedhere
\end{align*}
\end{proof}

Formula (\ref{eq:SymmetryKoopman}) in Proposition \ref{prop:Symmetry} is particularly useful when the domains of $U_\psi$ and $U_f$ are known. The easiest case is when both $U_\psi$ and $U_f$ induce bounded operators, i.e. when $D(U_\Psi)$ and $D(U_f)$ are the whole RKBS.

For sparsity, we follow a similar approach based on a specific sparsity pattern, i.e. factor systems\cite{OperatorTheoreticAspectsOfErgodicTheory} (see page 15 in the literature), and its application to Koopman operators.\cite{SparseKoopman}

\begin{definition}[Factor system]
    Let $f:X\rightarrow X$ be a discrete dynamical system on $X$ and $\RKBS$ be an RKBS. We call a triple $(Y,\Pi,F)$ a factor system if $Y$ is a set, $\Pi:X \rightarrow Y$ and $F:Y \rightarrow Y$ such that
    \begin{equation}
        \Pi \circ f = F \circ \Pi.    
    \end{equation}
\end{definition}

By similar arguments to the symmetry case, we get the following proposition.

\begin{proposition}\label{prop:Sparse}
    Let $f:X\rightarrow X$ be a discrete dynamical system, $(Y,\Pi,F)$ be a factor system, $\RKBS$ be an RKBS on $X$ and $(\B_Y,\B'_Y,\langle\cdot,\cdot\rangle_Y,k_Y)$ be an RKBS for $Y$. Let $K_\Pi:\SpanX \rightarrow \mathrm{Span}\{k_Y(y,\cdot): y \in Y\}$ defined by linear extension of $K_\Pi k(x,\cdot) := k_Y(\Pi(x),\cdot)$. Then
    \begin{equation}\label{eq:SparseIntertwine}
        K_\Pi K_f = K_F K_\Pi.
    \end{equation}
    For the Koopman operators $U_f$ and $U_F$ corresponding to $f$ and $F$ and $U_\Pi:D(U_\Pi) \rightarrow \B$ defined by $U_\Pi g:= g \circ \Pi$ on $D(U_\Pi) := \{g \in \B_Y: g \circ \Pi \in \B\}$ we have $U_fU_\Pi = U_\Pi U_F$ on
    \begin{equation}\label{eq:KoopmanSparseDomain}
        \{g \in \B: g \in D(U_\Pi) \cap D(U_F), D(U_\Pi) \in D(U_f), U_Fg \in D(U_\Pi)\}.
    \end{equation}
\end{proposition}

\begin{proof}
    The proof is similar to the proof of Proposition \ref{prop:Symmetry}. For all $x \in X$ we have
    \begin{align*}
        K_\Pi K_f k(x,\cdot) &= K_\Pi k(f(x),\cdot) = k(\Pi(f(x)),\cdot) = k(F(\Pi(x)),\cdot)\\
        &= K_FK_\Pi k(x,\cdot)
    \end{align*}
    and we get (\ref{eq:SparseIntertwine}). Similarly for the Koopman operator we have for all $g$ in the set given in (\ref{eq:KoopmanSparseDomain})
    \begin{equation*}
        U_fU_\Pi g = U_f \left(g\circ \Pi\right) = g \circ \Pi \circ f = g \circ F \circ \Pi = U_\Pi U_F g. \qedhere
    \end{equation*}
\end{proof}

The commutation and intertwining relations in Propositions \ref{prop:Symmetry} and \ref{prop:Sparse}, even though being similar, should be interpreted differently. For symmetries the commutation relation (\ref{eq:CommutingPerronFrob}) implies that the operators share eigenspaces - which can be exploited for dynamic mode decomposition (DMD).\cite{Symmetry, Perron-FrobSymmetry} 
Sparsity on the other hand intends to reduce the dynamical system to another (lower dimensional) one. If the object of interest can be fully observed by the factor system doing so is computationally beneficial if $Y$ is of lower dimension than $X$ and/or $F$ and $\Pi$ are not computationally complex.\cite{SparseKoopman}


\section{Conclusion}

We present a framework for Koopman and Perron-Frobenius operators on reproducing kernel Banach spaces which naturally includes the reproducing kernel Hilbert space situation. Due to the close relation between these operators, we deduce results for the Koopman operator based on the Perron-Frobenius operator and vice versa. We extended known results on general properties of Koopman and Perron-Frobenius operators on RKHS to RKBS and added new results on this topic concerning boundedness, closeness, and the domain of these operators.

We treat discrete time and continuous time dynamical systems. The Koopman (resp. Perron-Frobenius) operators for those systems share some common properties, but while for the discrete time case, the evolution operator $U_f$ (resp. $K_f$) is sufficient to describe the evolution of the system we turn to the
infinitesimal generator of the Koopman respectively Perron-Frobenius semigroup in the continuous time case. We investigated the domains of the infinitesimal generators and state a generator result based on a geometric condition on the kernel and the dynamics.

Since the notion of RKBS is very general we did not expect strong results, and it is clear that detailed investigations for specific RKBS -- such as Fock spaces, Hardy spaces, Sobolev spaces among others -- and dynamics remain a challenging task.\cite{IIS20, DKL17, CMS, CompOpAnalyticFunctions}
This is related to the problem of kernels adapted to the dynamics. We think that the power of the Koopman resp. Perron-Frobenius operator and the RKBS structure are fully released only if the RKBS is chosen according to the dynamics. At its core, this addresses the domain of the Koopman operator on the RKBS. A negative example of Koopman operators on RKBS can be obtained easily via Example \ref{example:RKHSRn} viewing $\R^n$ as the RKBS of linear forms and using non-linear dynamics $f$. This easily leads to the trivial domain $D(U_f) = \{0\}$ for the Koopman operator. An elaborate example of analytic dynamics on $\mathbb{C}$ states that under mild assumptions on a function space of entire functions, the Koopman operator can be bounded only if the dynamics are affine.\cite{IsaoBoundedCompOp}
Hence, we advertise a fruitful combination of Koopman theory and RKBS via kernels adapted to the dynamics; one such way is to consider invariant kernels (Remark \ref{rem:InvariantKernel}). This particular approach of invariant kernels leads to stronger results but is fairly restrictive.\cite{Das} Less restrictive examples -- but examples of kernels adapted to the dynamics -- are for example composition operators on the Hardy space for analytic dynamics or composition operators on Sobolev spaces for sufficiently regular dynamics.

We show that symmetry and sparsity transfer conceptually from well known cases to the RKBS setting. This yields certain commutation and intertwining relations for Koopman resp. Perron-Frobenius operators also in the RKBS setting. This aims at applications where those structures can be used to reduce computational complexity.

Future work might investigate applications of Koopman operators on RKBS to data science as for RKHS with applications to forecasting, (optimal) control, and stability analysis.\cite{Kawahara16, Rosenfeld, DasSpectrum}  We think that inspirations from the path of understanding the Koopman (or composition) operators on specific or general (reproducing) function spaces will give further insight into the RKBS situation.\cite{OperatorTheoreticAspectsOfErgodicTheory, DKL17}

\section*{Acknowledgements}
The authors want to greatly thank Igor Mezi\'{c}. The authors have met when they were independently visiting Igor Mezi\'{c}'s group. Their understanding of Koopman theory and several parts of this text benefited a lot from his deep insights into Koopman theory. We are thankful for meeting and discussing with Igor Mezi\'{c} and his group!

The first and second authors have been supported by  JST CREST Grant, Number JPMJCR1913, Japan.
The second author has been supported by JST ACTX Grant, Number JPMJAX2004, Japan.
The third author has been supported by European Union’s Horizon 2020 research and innovation programme under the Marie Skłodowska-Curie Actions, grant agreement 813211 (POEMA).

\appendix
\input{appendix_journal_chaos}

\nocite{*}
\bibliography{refs_journal}
\end{document}

%% file: appendix_journal_chaos.tex
\section{Further notions with respect to the bilinear form}


We devote this section to notions for RKBS $\RKBS$ concerned with the bilinear form $\langle \cdot,\cdot\rangle$. That includes the adjoint operator and hence plays an important rule for connecting the Koopman and the Perron-Frobenius operator.

\begin{definition}[Universal property]\label{Def:UnivProperty}
	We say $(\B,\B',\langle \cdot,\cdot\rangle)$ has the universal property if $\B$ embeds densely into $\C(X)$, by $i:\B \rightarrow \C(X)$, $g \mapsto g$.
\end{definition}

\begin{remark}\label{rem:i*Injective}
    The universal property states that the adjoint $i^*:\C(X) \rightarrow \B^*$ is injective.
\end{remark}

\begin{definition}[Annihilator]
	Let $\B$ and $\B'$ be Banach spaces and $\langle\cdot,\cdot\rangle$ be a continuous bilinear form on $\B \times \B'$. Let $A \subset \B$ then the annihilator $A^\perp$ of $A$ is defined by
	\begin{equation}\label{equ:AnnihilatorB}
		 A^\perp := \{ b' \in \B' : \langle a,b'\rangle = 0 \text{ for all } a \in A\} \subset \B'
	\end{equation}
	and for $A \subset \B'$ then the annihilator $A_\perp$ of $A$ is defined by
	\begin{equation}\label{equ:Annihilatorperp}
		 A_\perp := \{ b \in \B : \langle b,a\rangle = 0 \text{ for all } a \in A\} \subset \B.
	\end{equation}
\end{definition}

\begin{lemma}\label{lem:Aperp}
	Let $A \subset \B$ and $C \subset \B'$ then $A^\perp \subset \B$ and $C_\perp \subset \B'$ are closed. Further $(A^\perp)_\perp \supset A$ and $(C_\perp)^\perp \supset C$.
\end{lemma}

\begin{proof}
	By continuity of the bilinear form it follows that the set
	\begin{equation*}
		 A^\perp := \{ b' \in \B' : \langle a,b'\rangle = 0 \text{ for all } a \in A\} \bigcap\limits_{a \in A} (\langle a,\cdot\rangle)^{-1}(\{0\})
	\end{equation*}
	is closed. And similarly for $C$. For any element $a \in A$ we have by definition $\langle a,a^\perp\rangle$ for all $a^\perp \in A^\perp$, i.e. $a \in (A^\perp)_\perp$. A similar argument shows the corresponding statement for $C$.
\end{proof}

\begin{definition}[Hahn-Banach property]
	We say a triple $(\B,\B',\langle\cdot,\cdot\rangle)$ of Banach spaces $\B$ and $\B'$ and a continuous bilinear form $\langle\cdot,\cdot\rangle$ has the Hahn-Banach property if for all subspaces $A \subset \B$ we have $(A^\perp)_\perp = \overline{A}$.
\end{definition}

\begin{corollary}\label{cor:HBProp}
	If $(\B,\B',\langle\cdot,\cdot\rangle)$ has the Hahn-Banach property then $\B'$ is dense in $\B'$ with respect to $\langle\cdot,\cdot\rangle$.
\end{corollary}

\begin{proof}
	If $\langle b,b'\rangle = 0$ for all $b' \in \B'$ that means $b \in \B'_\perp$. Since $\B' = \{0\}^\perp$ we get by the Hahn-Banach property $b \in \B'_\perp = (\{0\}^\perp)_\perp = \overline{\{0\}} = \{0\}$, i.e. $b = 0$. That means $\B'$ is dense in $\B'$ with respect to $\langle\cdot,\cdot\rangle$ according to (\ref{equ:DenseB'}).
\end{proof}

%

\section{Adjoint operators}\label{app:AdjointOperators}
We recall the notion of an adjoint operator. Let $B$ and $B'$ be Banach spaces with continuous bilinear form $\langle \cdot,\cdot\rangle:\B \times \B' \rightarrow \mathbb{C}$ and $T$ be a densely defined (with respect to the bilinear form as in (\ref{equ:DenseB}) operator $T:B\supset D(T)\rightarrow B$. We call $T'$ the adjoint operator of $T$ (with respect to the bilinear from) if
\begin{equation}
    \langle Tx,y\rangle = \langle x,T'y\rangle
\end{equation}
for all $x \in D(T)$ and $y \in D(T')$ for
\begin{equation*}
    D(T') := \{y \in Y : \exists z \in Z \text{ with } \langle Tx,y\rangle = \langle x,z\rangle \text{ for all } x \in D(T)\}
\end{equation*}
We recall the notion of a closed operator. Let $X$ and $Y$ be Banach spaces. A operator $(A,D(A))$ for a linear subspace $D(A) \subset X$ with $A:D(A) \rightarrow Y$ is called closed (with respect to a topology $\mathcal{T}$, in this text that is the weak, weak* or norm topology) if
\begin{equation*}
    D(A) \ni x_n \rightarrow x \text{ and } Ax_n \rightarrow y \text{ implies } x \in D(A) \text{ and } Ax = y
\end{equation*}
where the limits are with respect to $\mathcal{T}$. The operator $A:\B \supset D(A) \rightarrow \B$ is called closed with respect to $\langle \cdot,\cdot\rangle$ if $D(A) \ni \langle x_n,z_1\rangle \rightarrow \langle x,z_1\rangle$ and $\langle Ax_n,z_2\rangle \rightarrow \langle y,z_2\rangle$ for all $z_1,z_2 \in X$ implies $x \in D(A)$ and  $Ax = y$. The definition of an operator $A:\B' \supset D(A) \rightarrow \B'$ being closed with respect to $\langle \cdot,\cdot\rangle$ is analogue. We call an operator $(A,D(A))$ closable if it has an extension $(\bar{A},D(\bar{A}))$ which is closed, i.e. $(\bar{A},D(\bar{A}))$ is a closed operator with $\bar{A} \supset A \text{ which denotes } D(A) \subset D(\bar{A}) \text{ and } \bar{A}x = Ax \text{ for all } x \in D(A)$.

We will state the following lemma which is well known for the case when $\langle \cdot,\cdot\rangle$ induces a Hilbert space.

\begin{lemma}\label{Lem:AdjointClosed}
    Let $T:B \supset D(T) \rightarrow B$ be a densely defined linear operator. Then $T'$ is uniquely defined and a closed (with respect to the norm topology as well with respect to $\langle\cdot,\cdot\rangle$) operator.
\end{lemma}

\begin{proof}
    For all $y \in D(T')$ we have for all $x \in D(T)$ that $\langle Tx,y\rangle = \langle x,T'y\rangle$. Since $D(T)$ is dense in $\B$ with respect to $\langle\cdot,\cdot\rangle$ the element $T'y$ is uniquely determined. To check closedness it suffices to show that $T'$ is closed with respect to $\langle\cdot,\cdot\rangle$ because the norm topology is stronger than the one induced by $\langle\cdot,\cdot\rangle$. Let $x_n \in D(T')$ with $\langle z_1,x_n\rangle \rightarrow \langle z_1,x\rangle$ and $\langle z_2,T'x_n\rangle \rightarrow \langle z_2,y\rangle$ for all $z_1,z_2 \in \B$. For all $v \in D(T)$ we have
    \begin{eqnarray}
    \langle v,y\rangle & \leftarrow & \langle v,T'x_n\rangle = \langle Tv,x_n\rangle \rightarrow \langle Tv,x\rangle.
    \end{eqnarray}
    Since $T$ is densely defined (with respect to $\langle\cdot,\cdot\rangle$) it follows $x \in D(T')$ and $T'x = y$.
\end{proof}

\begin{lemma}\label{lem:DenselyDefAdjoint}
    Let $T:B \supset D(T) \rightarrow B$ be a densely defined operator. If $T'$ is densely defined then $T$ is closable.
\end{lemma}

\begin{proof}
If $T'$ is densely defined we can build its adjoint with respect to the bilinear form $\langle \cdot,\cdot\rangle' : \B' \times \B \rightarrow \mathbb{K}$ defined by $\langle b',b\rangle := \langle b,b'\rangle$. Then the adjoint $T''$ of $T'$ is closed by Lemma \ref{Lem:AdjointClosed}. We claim that $T''$ is an extension of $T$. Let $x \in D(T)$, i.e. we have for each $y \in D(T')$ that $\langle Tx,y\rangle = \langle x,T'y\rangle = \langle T'y, x\rangle'$. But this exactly states that $x \in D(T'')$ with $T''x = Tx$. 
\end{proof}

The converse result is a classical result in Hilbert spaces. We present the adapted proof here in several steps.

We denote by $\Gamma(T) := \{(x,Tx) : x \in D(T)\}$ the graph of $T$. Further we define the following function $V$ on vector spaces $X,Y$.
\begin{equation}\label{equ:DefV}
    V:X \times Y \rightarrow Y \times X, V(x,y) := (-y,x).
\end{equation}

\begin{remark}
	Let $\B$ and $\B'$ be Banach spaces and $\langle\cdot,\cdot\rangle:\B \times \B' \rightarrow \mathbb{K}$ be a continuous bilinear form. Then we have a natural bilinear form $\langle\cdot,\cdot\rangle_s:(\B \times \B) \times (\B' \times \B') \rightarrow \mathbb{K}$ defined by
	\begin{equation}\label{DefBilinearSquare}
		\langle (b_1,b_2),(b_1',b_2')\rangle_s := \langle b_1,b_1'\rangle + \langle b_2,b_2'\rangle.
	\end{equation}
\end{remark}

\begin{lemma}\label{lem:GrVPerp}
    We have $\Gamma (T') = V(\Gamma (T)^\perp)$ with respect to the bilinear form $\langle\cdot,\cdot\rangle_s$.
\end{lemma}

\begin{proof}
    We have for $x,y \in B'$
    \begin{eqnarray*}
        (x,y) \in \Gamma (T') & \Leftrightarrow & \langle x,Tz\rangle = \langle y,z\rangle \text{ for all } z \in D(T)\\
        & \Leftrightarrow & \langle x,Tz\rangle - \langle y,w\rangle = 0 \text{ for all } (z,w) \in\Gamma (T)\\
        & \Leftrightarrow & \langle (x,y),(-w,z) \rangle = 0 \text{ for all } (z,w) \in \Gamma (T)\\
        & \Leftrightarrow & (x,y) \in V(\Gamma (T)^\perp)
    \end{eqnarray*}
\end{proof}

For Hilbert spaces the previous lemma is used in the proof of the statement that the adjoint of a closed densely defined operator is densely defined. The main argument makes use of the Hahn-Banach theorem which is particularly used to show that $\Gamma(T)^{\perp\perp} = \overline{\Gamma(T)}$, which is why we restrict to RKBS with the Hahn-Banach property. One of such RKBS are reflexive RKBS.

\begin{definition}[Reflexive RKBS]\label{def:RKBSReflexive}
	Let $\B$ and $\B'$ be Banach spaces and $\langle\cdot,\cdot\rangle$ be a continuous bilinear form on $\B \times \B'$. We call $(\B,\B',\langle\cdot,\cdot\rangle)$ a dual pairing if $\B'$ is isomorphic to $\B^*$ via the map $b' \mapsto \langle \cdot,b'\rangle$. We call $(\B,\B',\langle\cdot,\cdot\rangle)$ reflexive if it is a dual pairing and $\B$ is reflexive.
\end{definition}

\begin{remark}
    If the RKBS is given by an RKHS then the RKBS is reflexive according to Definition \ref{def:RKBSReflexive}.
\end{remark}

\begin{lemma}\label{lem:Reflexive}
	If $(\B,\B',\langle\cdot,\cdot\rangle)$ is a dual pairing then it has the Hahn-Banach property. If $(\B,\B',\langle\cdot,\cdot\rangle)$ is reflexive then $(\B',\B,\langle\cdot,\cdot\rangle')$ is reflexive for $\langle b',b\rangle' := \langle b,b'\rangle$, and $(\B \times \B,\B '\times \B',\langle\cdot,\cdot\rangle_s)$ is relfexive as well, where $\langle\cdot,\cdot\rangle_s$ denotes the bilinear form (\ref{DefBilinearSquare}).
	\end{lemma}

\begin{proof}
	To show the Hahn-Banach property it follows from Lemma \ref{lem:Aperp} that we only need to show $(A^\perp)_\perp \subset \overline{A}$ for any subspace $A$. Assume there exists a $x \in (A^\perp)_\perp \setminus \overline{A}$. By the Hahn-Banach theorem we can find a $b^* \in\B^*$ with $b^*\big|_{\overline{A}} = 0$ and $b^*(x) = 1$. If $(\B,\B',\langle\cdot,\cdot\rangle,k)$ is a dual pairing there exists a $b' \in \B'$ with $b^*(b) = \langle b,b'\rangle$ for all $b \in \B$. In particular we get for all $a \in A$ that $\langle a,b'\rangle = b^*(a) = 0$, i.e. $b^* \in A^\perp$. But since $x \in (A^\perp)_\perp$ it follows $0 = \langle x,b'
	\rangle$ which contradicts $\langle x,b'\rangle = b^*(x) = 1$. Now assume that $(\B,\B',\langle\cdot,\cdot\rangle,k)$ is reflexive. That means that $\B$ is reflexive and hence $\B'$ is reflexive. So it remains to show that the map
\begin{equation}\label{equ:EmbeddingBB'*}	
	e:\B \rightarrow (\B')^*, \; 
	b \mapsto \langle \cdot,b\rangle' = \langle b,\cdot\rangle \in (\B')^*
\end{equation}
is isomorphic. Due to boundedness of the bilinear form the map $e$ is bounded. Further it is injective due to the Hahn-Banach theorem and the assumption that $b' \mapsto \langle \cdot, b'\rangle$ is an isomorphism between $\B'$ and $\B^*$. It is also surjective by this reason, because due to the isomorphism between $\B'$ and $\B^*$ we have that $(\B')^* \cong \B^{**} \cong \B$ since $\B$ is reflexive. Hence let $y \in (\B')^*$, i.e. we can find a $b\in \B$ such that we can represent $y$ by $y = \langle b,\cdot\rangle = \langle \cdot,b\rangle'$, which is what remained to be shown. For the last statement note that we have $(V\times W)^* \cong V^* \times W^*$ for any topological vector spaces. So if $\B$ is reflexive so is $\B \times \B$ and any element $y \in (\B \times \B)^*$ can be written as $y = b_1^* + b_2^*$ for elements $b_1^*,b_2^* \in \B^*$, namely by $y(b_1,b_2) = y(b_1,0) + y(0,b_2) =: b_1^*(b_1) + b_2^*(b_2)$. Since $\B^* \cong \B'$ there exist elements $b_1',b_2' \in \B'$ with $b_i^* = \langle \cdot,b_i'\rangle$ for $i = 1,2$. In total we get
$y(b_1,b_2) = b_1^*(b_1) + b_2^*(b_2) = \langle b_1,b_1'\rangle + \langle b_2,b_2'\rangle = \langle (b_1,b_2),(b_1',b_2')\rangle_2$ for all $(b_1,b_2) \in \B$. So we see that $(b_1',b_2') \mapsto \langle \cdot,(b_1',b_2')\rangle$ defines a (continuous) surjective map from $\B'\times \B'$ to $(\B \times \B)^*$. Injectivity again follows from $\B$ being reflexive and the Hahn-Banach theorem.
\end{proof}

\begin{proposition}\label{prop:AdjointDenselyDefined}
	If $(\B,\B',\langle\cdot,\cdot\rangle)$ is reflexive and $T:D(T) \subset \B\rightarrow \B$ be a densely defined closable operator then $T'$ is densely defined.
\end{proposition}

\begin{proof}
	By Lemma \ref{lem:GrVPerp} we have $\Gamma(T') = (V\Gamma(T))^\perp$. Next we will show that $V$ commutes with the annihilator, i.e. $V(A_\perp) = (VA)_\perp$ for any $A \subset \B' \times \B'$. We have $(b_1,b_2) \in (VA)_\perp$ if and only if $0 = \langle (b_1,b_2),(a_1,a_2)\rangle$ for all $(a_1,a_2) \in A$, and we get
	\begin{eqnarray}\label{equ:VCommutesPerp}
		(b_1,b_2) \in (VA)_\perp & \Leftrightarrow & \langle (b_1,b_2),(-a_2,a_1)\rangle_s = 0 \; \; \; \forall (a_1,a_2) \in A\notag\\
		& \Leftrightarrow & 0 = -\langle b_1,a_2\rangle + \rangle b_2,a_1\rangle \; \; \; \forall (a_1,a_2) \in A\notag\\
		& \Leftrightarrow & 0 = \langle (-b_2,b_1),(a_1,a_2)\rangle\; \; \; \forall (a_1,a_2) \in A\notag \\
		& \Leftrightarrow & (-b_2,b_1) \in A_\perp\\
		& \Leftrightarrow & (b_1,b_2) \in V(A_\perp).  
	\end{eqnarray}
	Now assume $D(T')$ is not dense, by the Hahn-Banach theorem there exists a $b^* \in (\B')^* \setminus\{0\}$ with $b^* = 0$ on $D(T')$. Since $(\B,\B',\langle\cdot,\cdot\rangle)$ is reflexive there exists a $b\in \B\setminus \{0\}$ with $b^*(b') = \langle b,b'\rangle$ for all $b' \in \B$, in particular $\langle b,b'\rangle = 0$ for all $b' \in D(T')$, i.e. $b \in D(T')_\perp$. That implies $(b,0) \in \Gamma(T')_\perp$ and by Lemma \ref{lem:GrVPerp}
	\begin{align*}
		\Gamma(T')_\perp  &\overset{\text{Lemma \ref{lem:GrVPerp}}}{=} (V(\Gamma(T)^\perp))_\perp \overset{(\ref{equ:VCommutesPerp})}{=}V(\Gamma(T)^\perp_\perp)\\
		&\overset{\text{Lemma \ref{lem:Reflexive}}}{=} V(\overline{\Gamma(T)}).
	\end{align*}
	That shows that $(0,b) \in \overline{\Gamma(T)}$. Since $T$ is closable $\overline{\Gamma(T)}$ is contained in the graph $\Gamma(\bar{T})$ of the closure $\bar{T}$ of $T$, in particular we get $-b = \bar{T}0 = 0$, which contradicts $b \neq 0$.
\end{proof}

\section{Auxiliary lemma for measures}

The following Lemma was used in Proposition \ref{ElementaryPropertiesKoopmanRKHS} 5. to show that the Perron-Frobenius operator is not closed when the RKBS $\B$ on $X$ has the universal property and $X$ contains infinitely many points.

\begin{lemma}\label{AppendixLemmaLinearCombination}
    Let $X$ be compact and contain infinitely many elements. Then there exists an element in $\C(X)^*$ that is not a finite linear combination of dirac measures. In other words the set of dirac measures is not a basis of $\C(X)^*$.
\end{lemma}

\begin{proof}
    Let $(x_n)_{n \in \N} \subset X$ be a sequence of pairwise disjoint elements. Define
    \begin{equation}
        \mu := \sum\limits_{n = 1}^\infty \frac{1}{2^n} \delta_{x_n}.
    \end{equation}
    Since $\|\delta_{x_n}\| \leq 1$ the sum converges absolute and hence $\mu$ exists in $\C(X)'$.\\
    Assume $\mu$ is a finite linear combination of dirac measures. Then there exists a $m \in \N$ and $y_1,\ldots,y_m \in X$ and coefficients $a_1,\ldots,a_m \in \R$ with
    \begin{equation}\label{ProofAppendixMuDelta}
        \mu = \sum\limits_{k = 1}^m a_k \delta_{y_k}.
    \end{equation}
    Let $j_1 \in \N$ be the first index with $x_{j_1}\notin \{y_1,\ldots,y_m\}$. By Urysohn's lemma there exists a function $f \in \C(X)$ with the following properties $0\leq f \leq 1$, $f(y_k) = 0$ for all $k = 1,\ldots$ and $f(x_{j_1}) = 1$. Then we get the following contradiciton
    \begin{align*}
        0 &< \frac{1}{2^{j_1}} + \sum\limits_{n = j_1 + 1} \frac{1}{2^n} f(x_n) = \mu(f) = \left(\sum\limits_{k = 1}^m a_k \delta_{y_k}\right) (f)\\
        &=\sum\limits_{k = 1}^m a_k f(y_k) = 0. \qedhere
    \end{align*}
\end{proof}